\documentclass[12pt]{amsart}

\title[Linear programming bounds for hyperbolic surfaces]{Linear programming bounds for hyperbolic surfaces}

\author{Maxime Fortier Bourque}
\address{D\'epartement de math\'ematiques et de statistique, Universit\'e de Montr\'eal, 2920, chemin de la Tour, Montr\'eal (QC), H3T 1J4, Canada}
\email{maxime.fortier.bourque@umontreal.ca}

\author{Bram Petri}
\address{Institut de Math\'ematiques de Jussieu--Paris Rive Gauche ; UMR7586, Sorbonne Universit\'e - Campus Pierre et Marie Curie,
4, place Jussieu, 75252 Paris Cedex 05, France}
\email{bram.petri@imj-prg.fr}
\date{\today}
\date{\today}

\usepackage{amsmath,amsfonts,amssymb,amsthm,graphicx,overpic,tablefootnote}
\usepackage{float,enumitem}
\usepackage{color,xcolor,subcaption}
\usepackage[colorlinks=true,citecolor=blue]{hyperref}

\usepackage[margin=2.7cm]{geometry}  
\numberwithin{equation}{section}

\newtheorem{thm}{Theorem}[section]
\newtheorem{prop}[thm]{Proposition}

\newtheorem{cor}[thm]{Corollary}
\newtheorem{lem}[thm]{Lemma}
\newtheorem{conj}[thm]{Conjecture}

\newtheorem{innerthmrep}{Theorem}
\newenvironment{thmrep}[1]
  {\renewcommand\theinnerthmrep{#1}\innerthmrep}
  {\endinnerthmrep}

\theoremstyle{definition}

\newtheorem{defn}[thm]{Definition}

\theoremstyle{remark}
\newtheorem{rem}[thm]{Remark}
\newtheorem*{acknowledgements}{Acknowledgements}

\newcommand{\thmref}[1]{Theorem~\ref{#1}}

\newcommand{\propref}[1]{Proposition~\ref{#1}}
\newcommand{\secref}[1]{Section~\ref{#1}}

\newcommand{\lemref}[1]{Lemma~\ref{#1}}
\newcommand{\corref}[1]{Corollary~\ref{#1}}
\newcommand{\figref}[1]{Figure~\ref{#1}}

\newcommand{\tableref}[1]{Table~\ref{#1}}
\newcommand{\nc}{\newcommand}
\nc{\dmo}{\DeclareMathOperator}
\usepackage{dsfont}

\nc{\abs}[1]{\left| #1 \right|}
\nc{\bigO}[1]{O\left(#1\right)}
\nc{\card}[1]{\left|#1\right|}
\nc{\ceil}[1]{\left\lceil #1 \right\rceil}
\nc{\CC}{\mathbb{C}}
\nc{\dilog}{\mathcal{L}}
\nc{\floor}[1]{\left\lfloor #1 \right\rfloor}
\nc{\ind}{\mathds{1}}
\nc{\ZZ}{\mathbb{Z}}
\nc{\len}[1]{\left| #1 \right|}
\nc{\littleo}[1]{o\left(#1\right)}
\dmo{\Mat}{Mat}
\nc{\NN}{\mathbb{N}}
\nc{\norm}[1]{\left|\left| #1 \right|\right|}
\nc{\QQ}{\mathbb{Q}}
\nc{\RR}{\mathbb{R}}
\nc{\st}[2]{\left\{\, #1 \,:\, #2\,\right\}}
\dmo{\supp}{supp}
\nc{\tr}[1]{\mathrm{tr}\left(#1\right)}
\nc{\what}{\widehat}
\dmo{\im}{Im}
\dmo{\re}{Re}
\nc{\eps}{\varepsilon}
\dmo{\li}{li}

\dmo{\arccosh}{arccosh}
\dmo{\arcsinh}{arcsinh}
\dmo{\area}{area}
\dmo{\conv}{conv}
\dmo{\diam}{diam}
\dmo{\DD}{\mathbb{D}}
\dmo{\dist}{\mathrm{d}}
\nc{\HH}{\mathbb{H}}
\dmo{\Isom}{Isom}
\dmo{\MCG}{MCG}
\dmo{\MPL}{MPL}
\dmo{\Mod}{\mathcal{M}}
\dmo{\PL}{PL}
\nc{\Sphere}{\mathbb{S}}
\dmo{\sys}{sys}
\dmo{\kiss}{kiss}
\dmo{\Teich}{\mathcal{T}}
\nc{\Torus}{\mathbb{T}}
\dmo{\vol}{vol}
\dmo{\WP}{WP}
\nc{\Nsmall}{N_\mathrm{small}}
\dmo{\rect}{rect}

\dmo{\convTV}{\;\stackrel{\mathrm{TV}}{\longrightarrow}\;}
\nc{\ExV}[2]{\mathbb{E}_{#1}\left[#2\right]}
\dmo{\EE}{\mathbb{E}}
\nc{\Pro}[2]{\mathbb{P}_{#1}\left[#2\right]}
\dmo{\PP}{\mathbb{P}}
\nc{\distTV}[2]{\mathrm{d}_{\rm TV}\left(#1,#2\right)}
\dmo{\UU}{\mathbb{U}}
\nc{\Var}[2]{\mathbb{V}\mathrm{ar}_{#1}\left[#2\right]}

\dmo{\alt}{\mathfrak{A}}
\dmo{\Aut}{Aut}
\dmo{\Fix}{Fix}
\dmo{\GL}{GL}
\dmo{\Hom}{Hom}
\dmo{\id}{Id}
\dmo{\PSL}{PSL}
\dmo{\PGL}{PGL}
\dmo{\PO}{PO}
\dmo{\Rep}{Rep}
\dmo{\SL}{SL}
\dmo{\SO}{SO}
\dmo{\sym}{\mathfrak{S}}
\dmo{\inv}{\mathcal{I}}
\dmo{\orb}{\mathcal{O}}
\dmo{\stab}{Stab}

\dmo{\calA}{\mathcal{A}}
\dmo{\calB}{\mathcal{B}}
\dmo{\calC}{\mathcal{C}}
\dmo{\calD}{\mathcal{D}}
\dmo{\calE}{\mathcal{E}}
\dmo{\calF}{\mathcal{F}}
\dmo{\calG}{\mathcal{G}}
\dmo{\calH}{\mathcal{H}}
\dmo{\calI}{\mathcal{I}}
\dmo{\calJ}{\mathcal{J}}
\dmo{\calK}{\mathcal{K}}
\dmo{\calL}{\mathcal{L}}
\dmo{\calM}{\mathcal{M}}
\dmo{\calN}{\mathcal{N}}
\dmo{\calO}{\mathcal{O}}
\dmo{\calP}{\mathcal{P}}
\dmo{\calQ}{\mathcal{Q}}
\dmo{\calR}{\mathcal{R}}
\dmo{\calS}{\mathcal{S}}
\dmo{\calT}{\mathcal{T}}
\dmo{\calU}{\mathcal{U}}
\dmo{\calV}{\mathcal{V}}
\dmo{\calW}{\mathcal{W}}
\dmo{\calX}{\mathcal{X}}
\dmo{\calY}{\mathcal{Y}}
\dmo{\calZ}{\mathcal{Z}}

\begin{document}

\begin{abstract}
We adapt linear programming methods from sphere packings to closed hyperbolic surfaces and obtain new upper bounds on their systole, their kissing number, the first positive eigenvalue of their Laplacian, the multiplicity of their first eigenvalue, and their number of small eigenvalues. Apart from a few exceptions, the resulting bounds are the current best known both in low genus and as the genus tends to infinity. Our methods also provide lower bounds on the systole (achieved in genus $2$ to $7$, $14$, and $17$) that are sufficient for surfaces to have a spectral gap larger than $1/4$.  \\
\end{abstract}

\maketitle

\section{Introduction}

The goal of this paper is to prove new upper bounds on five invariants associated to a closed, oriented, hyperbolic surface $M$:
\begin{enumerate}
\item its \emph{systole} $\sys(M)$, the length of any shortest non-contractible closed curve in $M$;
\item its \emph{kissing number} $\kiss(M)$, the number of homotopy classes of oriented non-contractible closed curves of minimal length in $M$;
\item the \emph{first positive eigenvalue} $\lambda_1(M)$ (which coincides with the \emph{spectral gap}) of the Laplace--Beltrami operator $\Delta_M$ on $M$;
\item the \emph{multiplicity} $m_1(M)$ of the eigenvalue $\lambda_1(M)$, that is, the dimension of the corresponding eigenspace;
\item the number $\Nsmall(M)$, counting multiplicity, of \emph{small eigenvalues} of $\Delta_M$, that is, those contained in the interval $[0,1/4]$.
\end{enumerate}

We bound the first four invariants in terms of the genus of $M$ only, but the fifth one in terms of the genus and the systole. In low genus (or for small systole), our bounds are illustrated in Figures \ref{fig:plots1} and \ref{fig:plots2} (see also Tables \ref{table:sys} to \ref{table:gaps}). They beat all previous upper bounds except for the systole and kissing number in genus $2$, for $\lambda_1$ in genus $2$, $3$, $4$, and $6$, and for $\Nsmall$ when the systole is smaller than $2.317$. Note that $\Nsmall(M)<2$ if and only if $\lambda_1(M)>1/4$. We use this to show that there exist surfaces with a spectral gap larger than $1/4$ in genus $4$ to $7$, $14$, and $17$ (this was already known in genus $2$ and $3$). Whether such surfaces exist in every genus is a well-known open problem related to Selberg's eigenvalue conjecture \cite{Selberg} (see e.g. \cite[Question 1.1]{Mondal} and \cite[Problem 10.4]{Wright}). A lot of progress on this question in high genus was made recently \cite{covers,WuXue,LipnowskiWright,HideMagee,AnantharamanMonk, AnantharamanMonk2, MageePudervanHandel, polynomial_rate}.

\begin{figure}[H]

\begin{subfigure}{.48\textwidth}
\centering
\includegraphics[width=\textwidth]{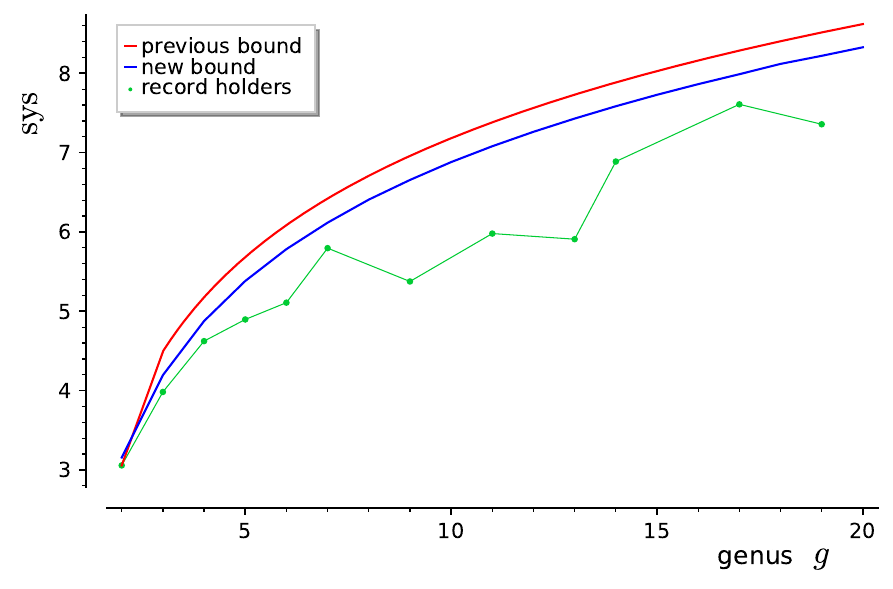}
\subcaption{Systole}
\label{fig:sys}
\end{subfigure}\hfill%
\begin{subfigure}{.48\textwidth}
\centering
\includegraphics[width=\textwidth]{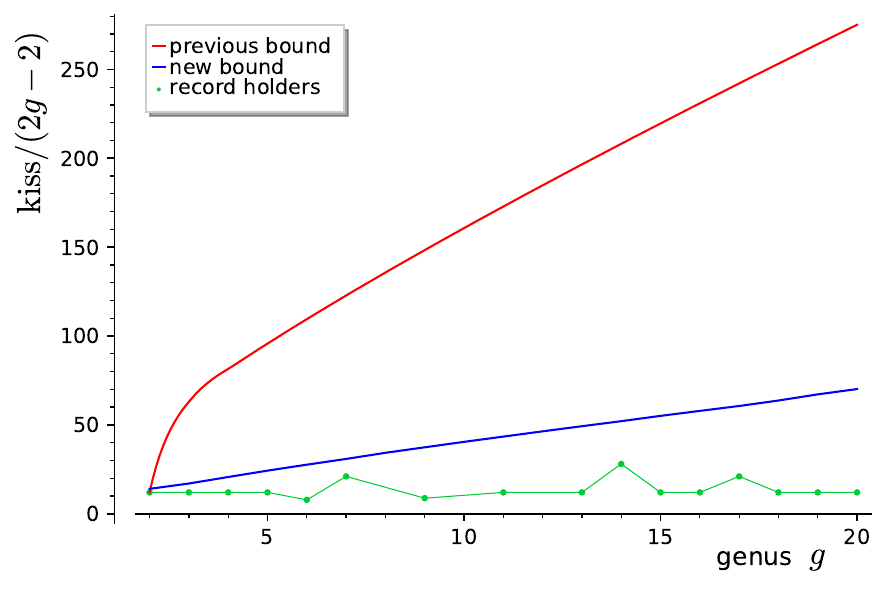}
\subcaption{Kissing number}
\label{fig:kiss}
\end{subfigure}\\

\begin{subfigure}{.48\textwidth}
\centering
\includegraphics[width=\textwidth]{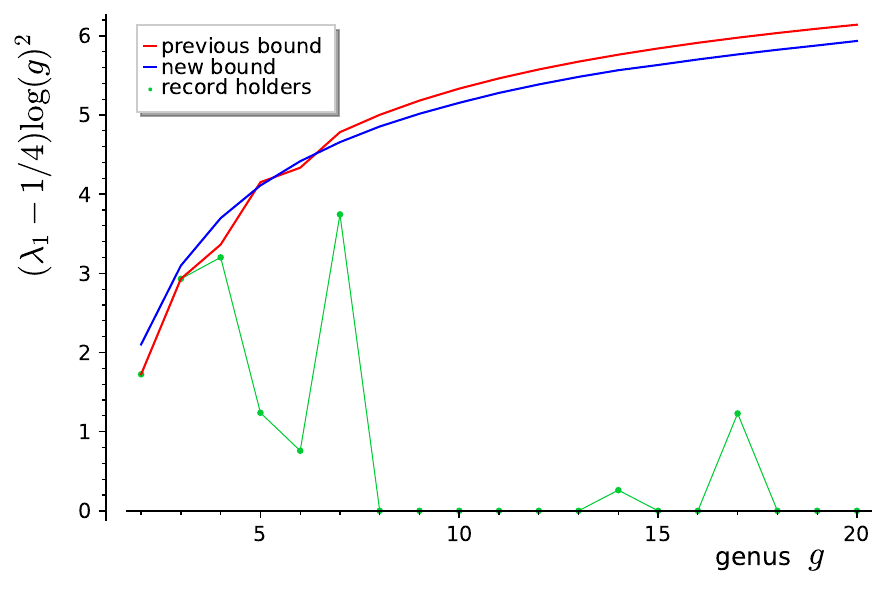}
\subcaption{First eigenvalue}
\label{fig:lam}
\end{subfigure}\hfill%
\begin{subfigure}{.48\textwidth}
\centering
\includegraphics[width=\textwidth]{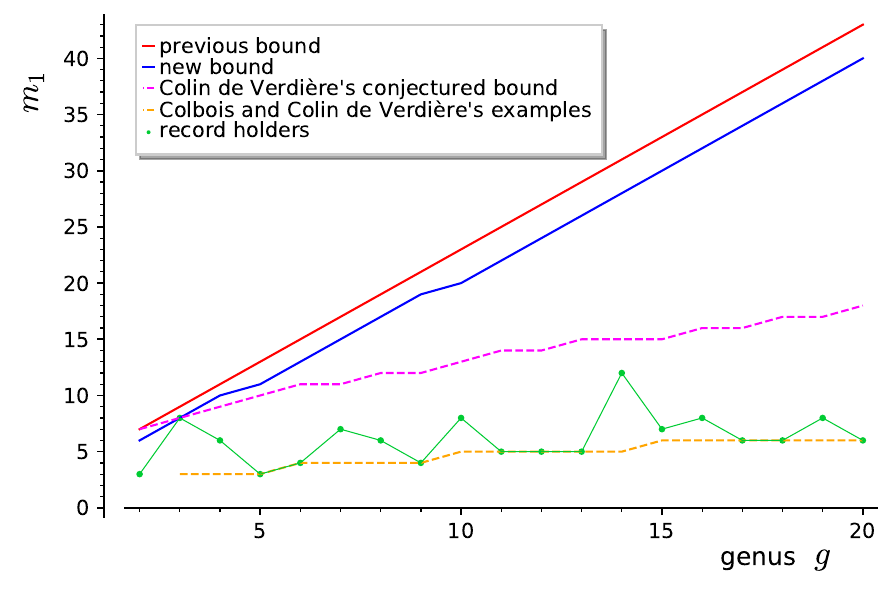}
\subcaption{Multiplicity}
\label{fig:mult}
\end{subfigure}\\

\caption{Upper bounds  and current record holders for the maximization of geometric and spectral invariants associated to hyperbolic surfaces.}
\label{fig:plots1}
\end{figure}

In higher genus (or for larger systole), our asymptotic bounds are as follows.

\begin{thmrep}{\ref*{thm:systole_asymp}}
There exists some $g_0 \geq 2$ such that every closed hyperbolic surface $M$ of genus $g\geq g_0$ satisfies
\[
\sys(M) < 2\log(g) + 2.409.
\]
\end{thmrep}

\begin{thmrep}{\ref*{thm:kiss_asymp}}
There exists some $g_0 \geq 2$ such that every closed hyperbolic surface $M$ of genus $g\geq g_0$ satisfies
\[
\kiss(M) < \frac{30.608 \cdot g^2}{\log(g)+1.2045}.
\]
\end{thmrep}

\begin{thmrep}{\ref*{thm:lamb_asymp}}
There exists some $g_0 \geq 2$ such that every closed hyperbolic surface $M$ of genus $g\geq g_0$ satisfies
\[
\lambda_1(M) < \frac14 + \left(\frac{\pi}{\log(g)+0.7436}\right)^2.
\]
\end{thmrep}

\begin{thmrep}{\ref*{thm:mult_asymp}}
There exists some $g_0 \geq 2$ such that every closed hyperbolic surface $M$ of genus $g\geq g_0$ satisfies
\[
m_1(M) \leq 2g - 1.
\]
\end{thmrep}

\begin{thmrep}{\ref*{thm:small_asymp}}
If $M$ is a closed hyperbolic surface of genus $g\geq 2$, then
\[
\Nsmall(M) <  \min\left( \frac{24 \pi^2(g-1)}{\sys(M)^3}, \frac{16(g-1)}{\sys(M)^2} \right).
\]
\end{thmrep}

These improve upon the previous best upper bounds established in \cite{Bavard}, \cite{KissingManifolds} (previously \cite{Parlier}), \cite{Cheng}, \cite{Sevennec}, and \cite{HuberSmallEigs} respectively. While the previous bounds used very different techniques from one invariant to another, our proofs are all based on the same method, namely, linear programming applied to the Selberg trace formula. Note that we have made no effort to estimate the required lower bound $g_0$ on genus in the above theorems, though this could  in principle be done by making certain calculations of limits effective.

\begin{figure}[htp]

\begin{subfigure}{.49\textwidth}
\centering
\includegraphics[width=\textwidth]{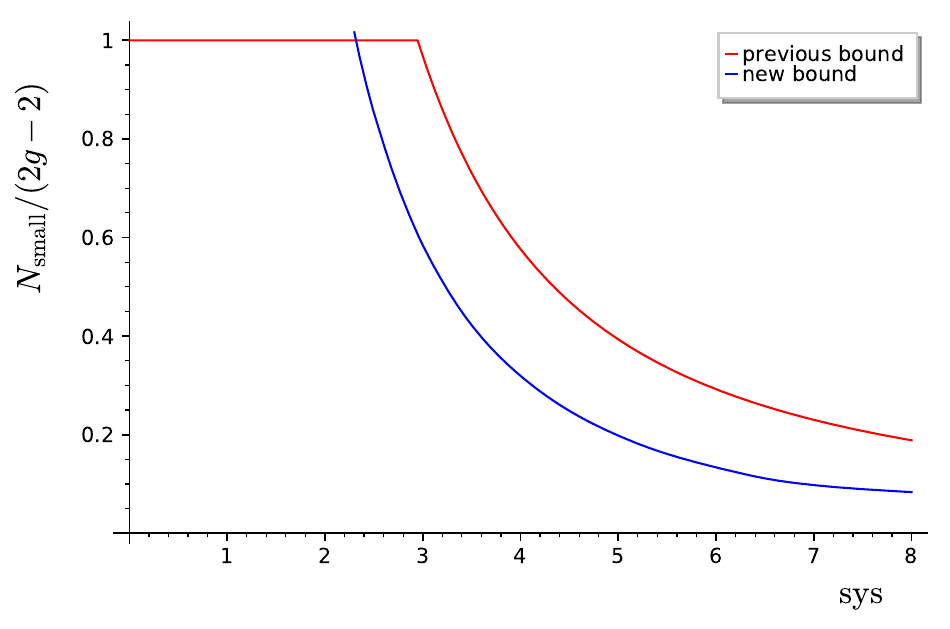}
\subcaption{Number of small eigenvalues. }
\label{fig:nsmall}
\end{subfigure}\hfill%
\begin{subfigure}{.49\textwidth}
\centering
\includegraphics[width=\textwidth]{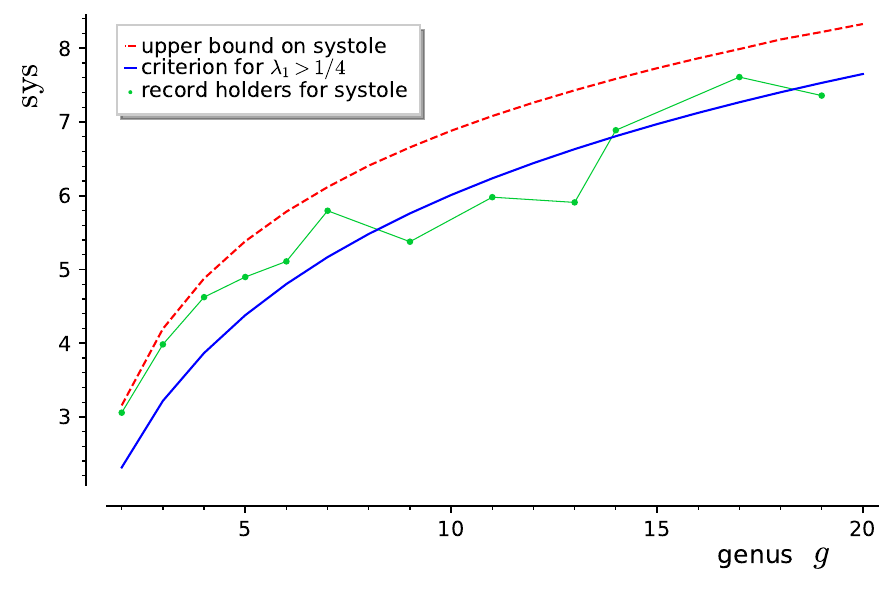}
\subcaption{Criterion for $\lambda_1>1/4$.}
\label{fig:gaps}
\end{subfigure}\\

\caption{Upper bounds on the number of small eigenvalues and lower bounds on the systole that imply a spectral gap larger than a quarter.}
\label{fig:plots2}
\end{figure}

In addition to finding new examples of surfaces with $\lambda_1 > 1/4$ as mentionned earlier, we also improve the previous best lower bounds of Colbois and Colin de Verdière \cite{CCV} on the maximum of $m_1$ in genus $4$, $7$, $8$, $10$, $14$ to $16$, and $19$. Note that Colin de Verdière's conjectured upper bound \cite{CdV86,CdV} on $m_1$ appearing in \figref{fig:mult} has since been disproved in genus 10 and 17 \cite{counterexamples}.

\subsection*{Context}

The invariants $\sys$, $\kiss$, $\lambda_1$, $m_1$, and $\Nsmall$ can be defined for any closed Riemannian manifold (with $1/4$ replaced by the bottom of the spectrum of the Laplacian on the universal cover of $M$ in the case of $\Nsmall$) and their ma\-xi\-mization has been studied by several authors for va\-rious classes of manifolds. These shape optimization problems seek to find optimal inequalities analogous to the isoperimetric inequality for other geometric quantities.

For example, if we fix a topological manifold $\Sigma$, then maximizing the systole over all Rieman\-nian manifolds $M$ of a given volume that are homeomorphic to $\Sigma$ is called the \emph{isosystolic problem} and it has been solved for the projective plane \cite{Pu}, the $2$-dimensional torus (Loewner), and the Klein bottle \cite{BavardKlein}. The maximum of $m_1$ is known for the same surfaces \cite{Besson,CdV,NadirashviliMultiplicite} as well as for the $2$-sphere \cite{Cheng}, while the maximum of $\lambda_1$ is known for the closed orientable surface of genus $2$ \cite{Jakobson,NayataniShoda} in addition to all the previous surfaces \cite{Hersch,LiYau,Nadirashvili,ElSoufi}. 

Another much-studied case is that of flat $d$-dimensional tori of unit volume. In that case, $\lambda_1(M)=(2\pi \sys(M^*))^2$ and $m_1(M)=\kiss(M^*)/2$ where $M^*$ is the torus dual to $M$, so the first four maximization problems above reduce to only two while the fifth is trivial since the bottom of the spectrum of the Laplacian on $\RR^d$ is $0$. Furthermore, if $\Lambda$ is a lattice such that $M \cong \RR^d / \Lambda$, then the balls of radius $\sys(M)/2$ centered at the points in $\Lambda$ form a sphere packing of density $2^{-d}\sys(M)^d \vol(B_1^d)$, where $B_1^d$ is the unit ball in $\RR^d$, and there are exactly $\kiss(M)$ balls tangent to any ball in this packing. In other words, maximizing the systole and kissing number of flat tori of unit area is equivalent to maximizing the packing density and kissing number among sphere packings whose centers form a lattice. Both problems have been solved in dimensions $1$ to $8$ and $24$ (see \cite[Table 1.1]{ConwaySloane} and \cite{CohnKumar}). In some cases, the solutions to these problems were obtained by a method known as \emph{linear programming} first introduced by Delsarte in the context of error-correcting codes \cite{Delsarte}. This was then adapted to prove bounds on the kissing numbers of arbitrary sphere packings (not necessarily coming from lattices) in \cite{DGS} and then on their packing density in \cite{CohnElkies}. In addition to giving optimal bounds in dimensions $8$ \cite{Viazovska} and $24$ \cite{CKMRV}, this approach also yields the best known asymptotic bounds on kissing numbers and packing density as the dimension tends to infinity \cite{CohnZhao}.

The five invariants we consider here were previously investigated for hyperbolic surfaces in \cite{SchmutzMaxima, BuserSarnak, Bavard}, \cite{SchmutzKissing, Schmutz43, Parlier}, \cite{Cheng, HideMagee, bootstrap}, 
\cite{CCV, LetrouitMachado}, and \cite{HuberSmallEigs,OtalRosas}, among many others. Prior to our work, the only optimal bounds known were for the systole \cite{Jenni} and kissing number \cite{SchmutzKissing} in genus $2$, and for $\Nsmall$ in every genus \cite{OtalRosas}. The known examples that ma\-xi\-mize $\Nsmall$ have a short pants decomposition \cite{Buser14}, which is why we are interested in improved bounds as the systole grows. We added one new optimal inequality to the list in \cite{Klein}, where we determined the maximum of $m_1$ in genus $3$ using the same techniques as in the current paper. Here we extend and improve this method in a systematic way. Although our new bounds in low genus do not appear to yield any new optimizers, they are a first step in that direction. The next step would be to compute the invariants for more examples.

\subsection*{Organization}

The paper is organized as follows. We start with preliminary sections on the Fourier transform, the Selberg trace formula, the linear programming method, and Bessel functions. This is followed by one section for each of the five invariants $\sys$, $\kiss$, $\lambda_1$, $m_1$, and $\Nsmall$. In each of these sections, we first present a general criterion for proving  upper bounds based on the Selberg trace formula. We then discuss the results that we have obtained from this criterion in low genus (or for small systole) using numerical optimization and conclude each section by proving an asymptotic bound. In both ranges, we compare our bounds with the previous best. The lower bounds on the systole that are sufficient to obtain a spectral gap larger than a quarter are described in subsection \ref{subsec:gaps} and the new examples with large $m_1$ are presented in subsection \ref{subsec:multiplicity_examples}.

\begin{acknowledgements}
We thank the CIRM for a research in residence during which part of this work was carried out. We also thank Mathieu Pineault for uncovering a mistake in one of our ancillary files and Chul-hee Lee for finding an error in a previous justification for the lower bound on the maximum of $m_1$ in genus $14$. These have since been fixed. BP is grateful for the ``Tremplins nouveaux entrants'' grant from Sorbonne Universit\'e, which allowed him to visit MFB at the Universit\'e de Montr\'eal.  MFB was partially supported by NSERC Discovery Grant RGPIN-2022-03649. Finally, we thank the anonymous referee for their careful reading and for suggesting much shorter proofs of Lemmas \ref{lem:behaviour_near_end} and \ref{lem:limit} with nicer resulting formulas.
\end{acknowledgements}

\section{The Fourier transform}

The Fourier transform of an integrable function $f:\RR \to \CC$ is defined by
\[
\what{f}(y) := \frac{1}{\sqrt{2\pi}}\int_\RR f(x)e^{-i y x} dx
\]
for $y\in \RR$. If $\what{f}$ is integrable, then the Fourier inversion theorem says that its Fourier transform is almost everywhere equal to $x \mapsto f(-x)$.

We will frequently use the scaling property that for $a>0$, the Fourier transform of $x \mapsto f(ax)$ is $y \mapsto \what{f}(y/a)/a$. If $f$ and $g$ are integrable, then 
\[
\what{f\ast g} =\sqrt{2\pi} \, \what{f} \, \what{g}
\]
where $\ast$ denotes the convolution and if $\what{f}$ and $\what{g}$ are also integrable, then
\[
\what{fg} = \frac{1}{\sqrt{2\pi}}\what{f} \ast \what{g}.
\]

When $f$ is an even function, which will be the case throughout the paper, its Fourier transform reduces to the cosine transform
\[
\what{f}(y) = \sqrt{\frac{2}{\pi}}\int_{0}^\infty f(x)\cos(yx) \,dx
\]
and is therefore even.

An continuous integrable function $f$ with integrable Fourier transform is said to be \emph{positive-definite} if $\what{f}(y) \geq 0$ for every $y \in \RR$. This is not the usual definition, but is equivalent to it under these hypotheses, by Bochner's theorem. If $f_1$ and $f_2$ are positive-definite (and continuous, integrable, with integrable Fourier transform), then $f_1 \ast f_2$ is positive-definite provided that $\what{f_1}\what{f_2}$ is integrable, and $f_1f_2$ is positive-definite provided that it is integrable.

\section{The Selberg trace formula}

Given a closed hyperbolic surface $M$  (always assumed to be oriented), we list the eigenvalues of its Laplace--Beltrami operator (acting on square-integrable functions) in non-decreasing order 
\[
0 = \lambda_0(M)<\lambda_1(M)\leq \lambda_2(M) \leq \cdots
\]
where each eigenvalue is repeated according to its multiplicity, i.e., the dimension of the corresponding eigenspace.

The set of oriented closed geodesics in $M$ is denoted by $\calC(M)$. This means that each unoriented closed geodesic appears twice in $\calC(M)$, once for each orientation. The length of a geodesic $\gamma \in \calC(M)$ is denoted $\ell(\gamma)$ and its \emph{primitive length} is denoted $\Lambda(\gamma)$. The latter is defined as the length of the shortest geodesic $\alpha$ such that $\gamma=\alpha^k$ for some power $k\geq 1$. The geodesic $\alpha$ is called \emph{primitive} because it cannot be expressed as a proper power of another geodesic.

A function $f:\RR \to \CC$ is said to be \emph{admissible} if it is even, continuous, integrable, and its Fourier transform $\what{f}$ is holomorphic in a horizontal strip of the form 
\[
\left\{ z \in \CC :| \im(z)| < \frac12+\eps \right\}
\] 
for some $\eps>0$ and satisfies the decay condition
\[
|\what{f}(z)| = O\left( \frac{1}{1+|z|^p} \right)
\]
for some $p>2$ in that strip. Note that the decay condition implies that $\what{f}$ and $y\what{f}(y)$ are integrable on the real line so that $f$ must be continuously differentiable by Fourier inversion and differentiation under the integral sign.

With the above notation and normalizations, the Selberg trace formula \cite[Section 9.5]{Buser} states that for every closed hyperbolic surface $M$ of genus $g$ and every admissible function $f:\RR \to \CC$, we have
\begin{equation*}
\sum_{j=0}^\infty \what{f}\left(\textstyle\sqrt{\lambda_j(M)-\frac14}\right) = 2(g-1)\int_0^\infty  \what{f}(x) x \tanh(\pi x) \,dx+ \frac{1}{\sqrt{2\pi}}\sum_{\gamma \in \calC(M)} \frac{\Lambda(\gamma)f(\ell(\gamma))}{2 \sinh(\ell(\gamma)/2)} .
\end{equation*}
Since $\what{f}$ is even, it does not matter which square root we use on the left-hand side. It is customary to write $r_j(M)$ for either of the two roots, so that $r_j(M)^2+\frac14 = \lambda_j(M)$. Note that our convention for the Fourier transform differs from the one used in \cite{Buser} by a factor of $1/\sqrt{2\pi}$, which explains the appearance of this factor in the above formula. 

\section{Linear programming}  \label{sec:CohnElkies}

Like the linear programming bounds of Cohn and Elkies \cite{CohnElkies} for the density of sphere packings, for each of the five invariants $\sys$, $\kiss$, $\lambda_1$, $m_1$, and $\Nsmall$, our criterion will take the following form:
\begin{quotation}
Suppose that $f$ is an admissible function such that $f$ and $\what{f}$ satisfy certain linear inequalities over certain intervals. Then $f$ and $\what{f}$ produce a bound on the given invariant that holds for every closed hyperbolic surface (satisfying certain conditions) in a given genus. 
\end{quotation} 

This is called ``linear programming'' because the inequalities are linear in the sense that any positive linear combination of functions that satisfy the inequalities still satisfies the inequalities. However, the function to be optimized (the resulting bound) is not linear in $f$. Moreover, the space we are optimizing over is infinite-dimensional and there are infinitely many inequalities to check (one at each point in the specified intervals). For these reasons, classical linear programming algorithms do not work well, which led Cohn and Elkies to devise the following strategy (adapted here to our setting). 

The idea is to consider functions $f$ of the form $f(x) = p(x^2) e^{-x^2/2}$ where $p$ is a polynomial. Such a function is automatically admissible since its Fourier transform takes the form $\what{f}(y)=q(y^2) e^{-y^2/2}$ for some polynomial $q$, hence defines an entire function with super-exponential decay in any horizontal strip. Moreover, the map $p \mapsto q$ is linear. In fact, it is diagonal with entries $(-1)^n$ with respect to the basis of generalized Laguerre polynomials $L_n^{(-1/2)}$, which in turn are related to the Hermite polynomials $H_{2n}$ by
\[
L_n^{(-1/2)}(x^2) = \frac{(-1)^n}{2^{2n} n!}H_{2n}(x).
\]
The upshot of the linearity of the map $p\mapsto q$ is that it is possible to impose linear equations on both $f$ and $\what{f}$ \emph{simultaneously}. All one has to do is solve a linear system of equations to find the coefficients of $p$ and $q$. The conditions we impose are that $f$ and $\what{f}$ have double zeros at certain points $x_1,\ldots , x_m$ and $y_1,\ldots,y_n$ respectively.

The reason for imposing double zeros is that it prevents local changes of sign and with enough double zeros at appropriate locations we are usually able to find some functions $f$ and $\what{f}$ that satisfy the required inequalities. This is done by hand through trial and error. Once we have found such suitable zeros, we try to wiggle them to decrease the resulting bound  (using standard optimization routines from \texttt{SciPy}), then add more double zeros to the right of the previous ones and repeat. When we increase the genus for a given invariant, we restart this procedure but take as input a subset of the double zeros found for the previous genus. 

For sphere packing bounds this optimization scheme over double zeros appears to converge quickly to a unique optimal function $f$ in each dimension. We have not found this to be the case for hyperbolic surfaces. One important difference is that for sphere packings, Cohn and Elkies assume that $f$ and $\what{f}$ have the same double zeros and the situation is fairly symmetric. This is not the case with the Selberg trace formula and the actual optimizers for our problems appear to have only finitely many zeros in some cases (but not their Fourier transform). Indeed, imposing more zeros for $f$ usually makes our bounds worse and the zeros have a tendency to fly off to infinity or collide when we run the optimizer.

The strategy we have described above is the one we use in low genus (or for small systole). In high genus (or for large systole), our asymptotic bounds are obtained by using special test functions related to Bessel functions and optimizing over certain parameters.

\subsection{Certifying inequalities on intervals}

Despite the numerical optimization used to produce our bounds, the end results are rigorous. The reason is that we work with rational zeros and polynomials with rational coefficients, so the linear systems involved are solved exactly over the rational numbers. The polynomials we get thus have actual double zeros rather than approximate ones.

To ascertain that $f(x)=p(x^2)e^{-x^2/2} \geq 0$ for every $x$ in a given interval $[a,b]$, we apply Sturm's theorem to count the number of distinct roots of $p$ in that interval and make sure that there are no more than the number of imposed double zeros. This implies that neither $p$ nor $f$ changes sign on the interval, and it then suffices to check that $p$ is strictly positive at some point or that its second derivative is strictly positive at a double zero.

We will also sometimes need to certify inequalities involving transcendental functions over intervals. In these cases, we approximate the transcendental functions with truncated Taylor series and apply Sturm's theorem to these approximation. The functions we consider either have positive  or alternating Taylor coefficients, allowing us to know if the approximations are from below or above.

In some cases, we need to find the minimum of a function $h(x)=r(x)e^{-x/2}$ on an interval $[a,b]$, where $r$ is a polynomial. Since $h'(x)=(r'(x)-r(x)/2)e^{-x/2}$ we can use Sturm's theorem to verify that $h$ has at most one critical point on the interval. If it has one, then it suffices to verify that $r'(a)-r(a)/2 >0$ and $r'(b)-r(b)/2 < 0$ so that the critical point is a local maximum. The minimum of $h$ is then at one of the endpoints, and this is also true if there is no critical point in the interval. 

\subsection{Certifying error bounds on integrals}

Another difference with sphere packing bounds is that we have inequalities involving integrals that need to be checked. For example, one of our bounds requires that
\begin{equation} \label{eqn:example_integral_ineq}
\what{f}(i/2) \geq 2(g-1)\int_0^\infty \what{f}(x) x\tanh(\pi x) \, dx.
\end{equation}
To obtain functions that satisfy this inequality, we first compute numerical approximations $I_n$ of the integrals 
\[
\int_0^\infty L_n^{(-1/2)}(x^2)e^{-x^2/2} x\tanh(\pi x) \, dx.
\]
In our linear system of equations, we then impose that $\what{f}(i/2)=q(-1/4)e^{1/8}$ is equal to $\rho$ times the numerical approximation of $2(g-1)\int_0^\infty \what{f}(x) x\tanh(\pi x) \, dx$ (given by a linear combination of the approximations $I_n$), where $\rho>1$ is some rational number. Technically, we also replace $e^{1/8}$ by a rational approximation in this equation.

Once we have found a good candidate function $f$, we verify a posteriori that inequality \eqref{eqn:example_integral_ineq} is satisfied. This is done by evaluating the left-hand side using interval arithmetic (which provides true lower and upper bounds on $\what{f}(i/2)$) and finding certified bounds on the integral.

For a function $h$ that is analytic in a neighborhood of a compact interval $[a,b]$, the \texttt{Arb} package \cite{Johansson} in \texttt{SageMath} \cite{sagemath} is able to compute the integral $\int_a^b h(x)\,dx$ with certified error bounds. However, improper integrals (and in particular infinite intervals) cannot be handled. We thus use the \texttt{Arb} package to estimate $\int_0^b \what{f}(x) x\tanh(\pi x) \, dx$ for some large $b$ and then estimate the remainder $\int_b^\infty \what{f}(x) x\tanh(\pi x) \, dx$ separately. For this, we use the inequalities 
\[
\tanh(\pi b) \leq \tanh(\pi x) \leq 1
\] for $x\geq b$. In all cases, our hypotheses will require that $\what{f}$ is eventually of constant sign, so it remains to estimate
\[
\int_b^\infty \what{f}(x) x \, dx = \int_b^\infty x q(x^2)e^{-x^2/2} \, dx.
\]
However, since $xq(x^2)$ is an odd polynomial, the function $x q(x^2)e^{-x^2/2}$ admits an explicit primitive and the integral can be computed exactly.

We will sometimes have to deal with more complicated integrals, in which case we estimate the remainder terms using ad hoc inequalities.

\subsection{Ancillary files}

Whenever we require certified error bounds on integrals in a proof, we explain how to estimate these integrals in the proof and state the resulting estimate that was obtained using interval arithmetic in \texttt{SageMath}. The calculations behind these estimates are all contained in the \texttt{Jupyter} notebook \texttt{certified\_integrals.ipynb} attached as an ancillary file to the arXiv version of this paper.

Then there is one file \texttt{verify\_invariant.ipynb} for each of the invariants we consider. Each such file contains a function \texttt{invariant\_poly} which computes a pair of poly\-no\-mials $(p,q)$ such that $f(x)=p(x^2)e^{-x^2/2}$ and $\what{f}(x)=q(x^2)e^{-x^2/2}$ are the Fourier transform of one another given a list of double zeros for each and perhaps additional data. Another function \texttt{invariant\_verify} checks that all the required conditions on $f$ and $\what{f}$ are satisfied and outputs a resulting rigorous upper bound on the invariant in question. The lists of input parameters that we used to produce the bounds in Tables \ref{table:sys} to \ref{table:gaps} are stored in various files \texttt{parameters\_invariant.sobj} that are loaded in the last cell of the \texttt{verify\_invariant} notebook. Upon execution of this last cell, the program runs the \texttt{invariant\_verify} function on each of these input parameters and prints out the resulting bounds.

\section{Bessel functions} \label{sec:Bessel}

Bessel functions were used in \cite{CohnElkies} to obtain a new proof of the second best asymptotic upper bound on the density of sphere packings in $\RR^n$ due to Levenshtein \cite{Lev}. We will also use these functions to obtain our asymptotic bounds. We list some of their properties here for later reference.

One of the many equivalent definitions \cite[p.40]{Watson} of the \emph{Bessel function of the first kind of order $\alpha$} is
\[
J_\alpha(z) := \left( \frac{z}{2}\right)^{\alpha}\sum_{n=0}^\infty \frac{(-1)^n}{n!\, \Gamma(n+\alpha+1)} \left( \frac{z}{2}\right)^{2n}
\]
when $\alpha$ is not a negative integer, where $\Gamma$ is the classical gamma function. For non-integer orders $J_\alpha$ is a multi-valued function, but by abuse of notation the quotient

\[
\frac{J_\alpha(z)}{z^\alpha} = \frac{1}{2^\alpha} \sum_{n=0}^\infty \frac{(-1)^n}{n!\, \Gamma(n+\alpha+1)} \left( \frac{z}{2}\right)^{2n}
\]
defines an even entire function that takes the value $2^{-\alpha}/\Gamma(\alpha+1)$ at the origin. Following \cite{Gorbachev}, this leads us to define the \emph{normalized Bessel function}
\[
\eta_\alpha(z) := \Gamma(\alpha+1) \sum_{n=0}^\infty \frac{(-1)^n}{n!\, \Gamma(n+\alpha+1)} \left(\frac{z}{2}\right)^{2n} = 2^\alpha \Gamma(\alpha+1) \frac{J_\alpha(z)}{z^\alpha}
\]
satisfying $\eta_\alpha(0)=1$. For $\alpha > -1/2$, Poisson's integral formula for Bessel functions \cite[p.224, eq. (10.9.4)]{NIST} can be written as
\[
\eta_\alpha(x) = \frac{2}{B(\frac12,\alpha+\frac12)} \int_0^1 (1-t^2)^{\alpha-1/2} \cos(xt)\,dt
\]
where
\[
B(a,b) = \int_0^1 t^{a-1}(1-t)^{b-1} = \frac{\Gamma(a)\Gamma(b)}{\Gamma(a+b)}
\]
is the Beta function (note that there is a factor of 2 missing in \cite[eq. (2.8)]{Gorbachev}). This means that $\eta_\alpha$ is the Fourier transform of 
\[
\chi_\alpha(t) =  \frac{\sqrt{2\pi}}{B(\frac12,\alpha+\frac12)} \rect(t/2) (1-t^2)^{\alpha - 1/2}
\]
where $\rect$ is the characteristic function of the interval $\left[-\frac12,\frac12\right]$. By Fourier inversion, we have
\[
\what{\eta_\alpha}(t) =  \chi_\alpha(t)
\]
whenever $\alpha>1/2$, which is when $\eta_\alpha$ is integrable. In particular, $\eta_\alpha$ is positive-definite for $\alpha > 1/2$ and its Fourier transform is supported in $[-1,1]$. By the easy direction of the Paley--Wiener theorem, this implies that $\eta_\alpha$ has exponential type $1$. In fact, along the imaginary axis we have the following exact asymptotic for every $\alpha \geq -1/2$ \cite[p.203]{Watson}:
\begin{equation} \label{eq:Bessel_assymptotic_imag}
J_\alpha(ix) \sim \frac{e^{x}}{\sqrt{2\pi x}} \quad \text{as }  x \to \infty \text{ in }\RR. 
\end{equation}

The above integrability condition on $\psi_\alpha$ follows from the asymptotic formula
\begin{equation} \label{eq:Bessel_assymptotic_real}
J_\alpha(z) = \sqrt{\frac{2}{\pi z}} \left( \cos \left( z - \frac{\alpha \pi}{2} - \frac{\pi}{4} \right) + O(e^{|\im z|}/|z|) \right)
\end{equation}
as $|z| \to \infty$ with $|\arg z| < \pi$ \cite[p.364]{Handbook}.
Also note that $J_\alpha(x)$ vanishes to order $\alpha$ at the origin, so that for $\alpha \geq -1/2$ the function $x\mapsto \sqrt{x}J_\alpha(x)$ is bounded near the origin and hence on $(0,\infty)$ by continuity and the above asymptotic. 

We will frequently make use of the even entire functions
\[
\varphi_\alpha(z)=\frac{J_\alpha(z/2)^2}{z^{2\alpha}} = \left( \frac{\eta_\alpha(z/2)}{4^\alpha \Gamma(\alpha+1)} \right)^2
\]
and
\[
\psi_\alpha(z)=\frac{J_\alpha(z)^2}{z^{2\alpha}(1- (z/j_\alpha)^2)}
\]
where $j_\alpha$ is the first positive root of $J_\alpha$. These are such that $\varphi_\alpha(x) \geq 0$ for every $x \in \RR$ and $\psi_{\alpha}(x)\leq 0$ for $|x|\geq j_\alpha$. Up to positive constants, $\varphi_\alpha(2x)$ is equal to $\what{\chi_\alpha \ast \chi_\alpha}(x)$ so its Fourier transform is a positive constant multiple of $\chi_\alpha \ast \chi_\alpha \geq 0$ as long as $\varphi_\alpha$ and $\chi_\alpha$ are integrable, which holds whenever $\alpha>0$. In other words, $\varphi_\alpha$ is positive-definite if $\alpha >0$, with Fourier transform supported in $[-1,1]$. It was also shown in \cite[Remark 1.1]{Gorbachev} that $\psi_\alpha$ is positive-definite if $\alpha \geq -1/2$. Observe that $\what{\varphi_\alpha}$ is admissible if $\alpha > 1/2$ and $\what{\psi_\alpha}$ is admissible if $\alpha>-1/2$ by the asymptotic formula \eqref{eq:Bessel_assymptotic_real}.

\section{Systole}\label{sec:Systole}

\subsection{The criterion}
The systole of a closed hyperbolic surface is defined as the length of any of its shortest closed geodesics (also called systoles).  Our criterion for bounding the systole goes as follows.

\begin{thm} \label{thm:systole}
Let $g \geq 2$. Suppose that $f$ is a non-constant admissible function for which there exists an $R>0$ such that
\begin{itemize}
\item $f(x)\leq 0$ if $x \geq R$;
\item $\what{f}(\xi) \geq 0$ for every $\xi \in \RR \cup i\left[-\frac12,\frac12\right]$;
\item $\what{f}(i/2) \geq 2(g-1)\int_0^\infty \what{f}(x) x \tanh(\pi x) \,dx$.
\end{itemize}
Then $\sys(M) \leq R$ for every closed hyperbolic surface $M$ of genus $g$. 
\end{thm}
\begin{proof}
Suppose that there is a hyperbolic surface $M$ of genus $g$ such that $\sys(M)> R$. Then $\sys(N)>R$ for every surface $N$ in some connected neighborhood $U$ of $M$ in moduli space. This implies that  $f(\ell(\gamma)) \leq 0$ for every $\gamma \in \calC(N)$ and we also have $\what{f}\left(\textstyle\sqrt{\lambda_j(N)-\frac14}\right)) \geq 0$ for every $j\geq 0$ by the hypotheses on $f$ and $\what{f}$. From the Selberg trace formula, we obtain
\begin{align*}
\what{f}(i/2) &\leq \sum_{j=0}^\infty \what{f}\left(\textstyle\sqrt{\lambda_j(N)-\frac14}\right)\\
 & = 2(g-1)\int_0^\infty \what{f}(x) x \tanh(\pi x) \,dx + \frac{1}{\sqrt{2\pi}}\sum_{\gamma \in \calC(N)} \frac{\Lambda(\gamma) f(\ell(\gamma))}{2 \sinh(\ell(\gamma)/2)} \\
&\leq 2(g-1)\int_0^\infty \what{f}(x) x \tanh(\pi x) \,dx \\
&\leq \what{f}(i/2)
\end{align*}
for every $N \in U$. We conclude that $\what{f}\left(\textstyle\sqrt{\lambda_j(N)-\frac14}\right)=0$ for every $j\geq 1$. Since $\what{f}$ is holomorphic in a strip and not constant equal to zero, its zeros are isolated. This implies that for every $j\geq 1$, the eigenvalue $\lambda_j(N)$ is a constant function of $N \in U$ since eigenvalues depend continuously on the metric (see e.g. \cite{Bando}). Therefore, all the surfaces in $U$ are isospectral. However, Gel'fand proved that any continuous deformation of $M$ that preserves the entire Laplace spectrum is constant \cite{Gelfand}, which is a contradiction.
\end{proof}

\begin{rem}
The analogous result for flat tori was proved in \cite[Theorem 3.2]{CohnElkies} using a rescaling and limiting argument for the second half of the proof. 
\end{rem}

\begin{rem}
It is easy to see that if the inequality in the third bullet point is strict, then the conclusion can be strengthened to a strict inequality. The proof proceeds similarly as above, but the chain of inequalities directly leads to a contradiction.
\end{rem}

\subsection{Low genus}

The upper bounds we have obtained from \thmref{thm:systole} though numerical optimization are listed in Table \ref{table:sys}  for $2\leq g \leq 20$. The verification of these values is done in the ancillary file \texttt{verify\_systole.ipynb}. They are lower than the previous best upper bounds except in genus $2$ where the optimal bound is $2\arccosh(1+\sqrt{2})\approx 3.057142$ \cite{Jenni}. In all other genera, the previous best upper bound was Bavard's inequality \cite{Bavard}
\begin{equation} \label{eq:Bavard}
\sys(M) \leq 2\arccosh\left( \frac{1}{2\sin(\pi / (12g-6))} \right),
\end{equation}
which comes from a sharp upper bound on the radius of an embedded disk in $M$.

We have also listed the largest recorded value of the systole in some genera. Those listed in genus $7$, $14$, and $17$ come from Hurwitz surfaces (these are $(2,3,7)$-triangle surfaces as defined in \secref{subsec:multiplicity_examples}). Technically, the values from \cite{VogelerThesis} and \cite{ScheinShoan} were obtained by numerical calculations in triangle groups and are not completely rigorous, but they could be made rigorous in principle (this was done in \cite{DT} for the Klein quartic and in \cite{Woods} for the Hurwitz surfaces of genus $14$). For Hurwitz surfaces, the calculations from \cite{ScheinShoan} corroborate those of \cite{VogelerThesis}.

Since the systole does not decrease under covers, one could fill in all the blanks in the table with values in lower genera. Similarly, the value listed in genus $13$ persists in every genus $g >13$ \cite{FBRafi}. We decided not to list these since better constructions surely exist.

{\footnotesize
\begin{table}[htp] 
\centering
\caption{Bounds on the maximum of $\sys$. The dagger indicates where the linear programming bound fails to beat the previous best upper bound.} \label{table:sys}
\begin{tabular}{|l|l|l|l|} 
\hline
genus & lower bound  & LP bound & previous upper bound \\
\hline
2 & 3.057141 \cite{Jenni} & $3.156053^\dagger$ &  3.057142 \cite{Jenni} \\   
3 & 3.983304 \cite{SchmutzMaxima} & 4.194719 & 4.494373  \cite{Bavard} \\
4 & 4.624499 \cite{SchmutzMaxima} & 4.876863 & 5.176481 \cite{Bavard} \\
5 & 4.91456 \cite{SchmutzKissing} & 5.381937 & 5.682841 \cite{Bavard}\\
6 &  5.109 \cite{Casamayou}        & 5.783671 & 6.086062 \cite{Bavard}\\
7 & 5.796298 \cite{VogelerThesis} & 6.117160 & 6.421249  \cite{Bavard} \\
8 &          & 6.407734 & 6.708126 \cite{Bavard} \\
9 &   5.376 \cite{SchmutzMaxima}      & 6.655635 & 6.958903 \cite{Bavard} \\
10 &         & 6.880869 & 7.181671 \cite{Bavard}  \\
11 & 5.980406 \cite{SchmutzMaxima} & 7.080715  & 7.382068 \cite{Bavard} \\
12 &          & 7.262735 & 7.564184 \cite{Bavard} \\
13 & 5.909039 \cite{FBRafi}  & 7.429527 & 7.731080 \cite{Bavard} \\
14 & 6.887905 \cite{Woods,VogelerThesis} & 7.584859 &  7.885106 \cite{Bavard} \\
15 &          & 7.729299 & 8.028108 \cite{Bavard} \\
16 &          & 7.863529 & 8.161558 \cite{Bavard} \\
17 & 7.609407 \cite{VogelerThesis} & 7.988773 & 8.286655 \cite{Bavard}\\
18 &   & 8.118854 & 8.404383 \cite{Bavard}\\
19 & 7.358 \cite{ScheinShoan} & 8.220710 & 8.515562 \cite{Bavard}\\
20 &  & 8.328393 & 8.620882 \cite{Bavard}\\
\hline
\end{tabular}
\end{table}
}

\subsection{Asymptotics}  \label{subsec:systole_asymptotic}

Note that the term $\sin(\pi / (12g-6))$ appearing in Bavard's bound is asymptotic to $\frac{\pi}{12 g}$ as $g \to \infty$ and \[\arccosh(x)=\log(x+\sqrt{x^2 - 1}) = \log(x)+\log(2) + o(1)\] as $x \to \infty$ so that Bavard's bound \eqref{eq:Bavard} can be rewritten as
\begin{align*}
\sys(M) &\leq 2 \log\left( \frac{6g}{\pi}\right)+2\log(2)+o(1) \\
&= 2\log(g) + 2\log\left( \frac{12}{\pi} \right) + o(1) \\
&= 2\log(g) + 2.680353\ldots + o(1)
\end{align*}
as $g \to \infty$. By comparison, the elementary area bound coming from the fact that a disk of radius $\sys(M)/2$ is embedded is
\begin{align*}
\sys(M) &\leq 4 \arcsinh(\sqrt{g-1}) \\
&= 2\log(g) + 4 \log(2) + o(1) \\
&= 2\log(g) + 2.772588\ldots + o(1)
\end{align*}
as $g \to \infty$.

We will decrease the additive constant in Bavard's bound by roughly $0.271$, which is consistent with the improvement we have observed in small genus. 

\begin{thm} \label{thm:systole_asymp}
There exists some $g_0 \geq 2$ such that every closed hyperbolic surface $M$ of genus $g\geq g_0$ satisfies
\[
\sys(M) < 2\log(g) + 2.409.
\]
\end{thm}

\begin{rem}
In terms of area, this result means that in large genus, a disk of radius $\sys(M)/2$ cannot occupy a proportion of more than
\[
\frac{e^{2.409/2}}{4} \approx 0.833772\ldots
\]
of $M$, while a maximal embedded disk can occupy as much as 
\[
\frac{3}{\pi} \approx 0.954929\ldots
\]
of the surface by Bavard's result \cite{Bavard}.
\end{rem}

\begin{rem}
In large genus, the constructions with the fastest growing systole known are given by towers of principal congruence covers of arithmetic surfaces. For each such tower, there is a constant $c$ such that
\[
\sys(M) \geq \frac43 \log(g) - c 
\]
for every surface $M$ of genus $g$ in the tower \cite[Theorem 1.5]{KSV}.
\end{rem}

The lengthy proof of \thmref{thm:systole_asymp} will require several estimates presented in the form of lemmata below. The strategy is to apply \thmref{thm:systole} with functions $f$ such that
\[
\what{f}(x) = h_{c}(bx) \varphi_\alpha(Rx),
\]
for some parameters $\alpha$, $b$, $c$ and $R$, where 
\[
h_{c}(x)= (c-1+x^2) e^{-x^2/2} \quad \text{and} \quad \varphi_\alpha(x) = \frac{J_\alpha(x/2)^2}{x^{2\alpha}}
\]
 with $J_\alpha$ the Bessel function of order $\alpha$ as in \secref{sec:Bessel}. We were led to these types of functions by studying the numerical data gathered in small genus. We believe they are nearly optimal. Using functions of the form $\what{f}(x) = \varphi_\alpha(Rx)$ instead yields a bound of the form $2\log(g) + c$ but with a $c$ larger that Bavard's, which is why we need to use more complicated functions.
 
The parameters $\alpha, b, c$ will be fixed at some point and only $R$ will depend on the genus. Since we need 
\[
\what{f}(i/2) > 2(g-1) \int_0^\infty \what{f}(x) x \tanh(\pi x) \, dx \geq 0,
\]
 we require that
\[
c-1-b^2/4 > 0  \quad \text{or equivalently} \quad c > 1 + b^2/4,
\] 
which in turn implies that $\what{f}$ is non-negative on $\RR \cup i [-\frac12,\frac12]$.

To apply \thmref{thm:systole}, we need to show that $f$ is eventually negative and to estimate the location of its last sign change. By Fourier inversion and the convolution theorem, we have
\begin{equation} \label{eqn:f_sys}
f = \frac{1}{\sqrt{2\pi}} \, \what{h_{c}^b}\ast \what{\varphi_\alpha^{R}}
\end{equation}
where $h_{c}^b(x)=h_c(bx)$ and $\varphi_\alpha^R(x) = \varphi_\alpha(Rx)$. Recall that $\varphi_\alpha$ is positive-definite with Fourier transform supported in $[-1,1]$ provided that $\alpha>0$, as explained in \secref{sec:Bessel}. It follows that $\what{\varphi_\alpha^{R}}(x) = \what{\varphi_\alpha}(x/R)/R$ is non-negative and supported in $[-R,R]$ while
\[
\what{h_{c}^b}(x) =\frac1b \what{h_{c}}(x/b) = \frac{1}{b}(c-x^2/b^2) e^{-x^2/(2b^2)}
\]
is non-positive outside $[-b\sqrt{c},b\sqrt{c}]$. From this, it is easy to show that the convolution $f$ is non-positive outside $[-R-b\sqrt{c},R+b\sqrt{c}]$. However, $f(R+b\sqrt{c})<0$ so that the last sign change occurs before that. Here is how we can estimate its location more precisely.

\begin{lem} \label{lem:pointwise_criterion}
Let $f_R$ be defined as in \eqref{eqn:f_sys} with $\alpha \in (0,1)$, $b>0$, and $c > 0$. If $\kappa_n\to \kappa_0 \in \RR$ and $R_n\to \infty$ as $n\to \infty$, then $R_n^{2\alpha+1} f_{R_n}(R_n+b\kappa_n)$ converges to
\[
\mu_{\alpha,b,c,\kappa_0}:= \frac{b^{2\alpha}}{\pi \, \Gamma(2\alpha+1)} \int_{\kappa_0}^\infty (x - \kappa_0)^{2\alpha} (c-x^2) e^{-x^2/2} \, dx
\]
as $n\to \infty$ and this limit depends continuously on the parameters. In particular, if $\alpha$, $b$ $c$ and $\kappa_0$ are such that $\mu_{\alpha,b,c,\kappa_0}<0$, then $f_{R_n}(R_n+b\kappa_n)$ is negative whenever $n$ is large enough.
\end{lem}

The proof will require the following intermediate lemma.

\begin{lem}  \label{lem:behaviour_near_end}
For every $\alpha \in (0,1)$, there exist a constant $d_\alpha > 0$ such that
\[
R^{2\alpha} \what{\varphi_\alpha}\left(1-\frac{y}{R}\right) \leq d_\alpha y^{2\alpha}
\]
for every $R>0$ and $y> 0$. Moreover,
\[
\lim_{R\to \infty} R^{2\alpha} \what{\varphi_\alpha}\left(1-\frac{y}{R}\right) = c_\alpha y^{2\alpha}
\]
 for every $y > 0$, where $c_\alpha=\sqrt{\frac{2}{\pi}} \frac{1}{\Gamma(2\alpha+1)}$.
\end{lem}
 
\begin{proof} The following proof is due to an anonymous referee; it considerably shortens our original proof and provides a simpler formula for $c_\alpha$. The idea is to prove that the function $\Phi_\alpha(t) = \what{\varphi_\alpha}(t)/(1-t)^{2\alpha}$ defined for $t \in (-\infty, 1)$ extends continuously at $t=1$.

As was mentioned in Section \ref{sec:Bessel},
\[
\what{\varphi_\alpha} = \frac{1}{16^\alpha \Gamma(\alpha+1)^2} \calF[x \mapsto \eta(x/2)^2] = \frac{4}{\sqrt{2\pi}\cdot 16^\alpha \Gamma(\alpha+1)^2} \calF[\tau_\alpha\ast \tau_\alpha]
\]
where $\calF$ is the Fourier transform and $\tau_\alpha(x) = \chi_\alpha(2x)$. Filling this in and using that $\Gamma(\frac{1}{2})=\sqrt{\pi}$ yields
\[
\what{\varphi_\alpha}(t) = \frac{1}{2^{4\alpha-5/2}\pi^{1/2}\Gamma(\alpha+\frac{1}{2})^2} \int_{-\frac{1}{2}+t}^{\frac{1}{2}} (1-4(t-y)^2)^{\alpha-\frac{1}{2}} (1-4y^2)^{\alpha-\frac{1}{2}} dy
\]
for every $t \in [0,1]$. We then substitute $u = \frac{y-t+\frac{1}{2}}{1-t}$ and obtain, after elementary simplifications,
\[
\what{\varphi_\alpha}(t) =\sqrt{\frac{2}{\pi}} \frac{(1-t)^{2\alpha}}{\Gamma(\alpha+\frac{1}{2})^2} \int_0^1 [u(1-u)(1-(1-t)u)(t+(1-t)u)]^{\alpha-1/2} du.
\]
Since the integrand is bounded by a constant times the integrable function $[u(1-u)]^{\alpha-1/2}$ as $t$ approaches $1$, the dominated convergence theorem implies that
\begin{align*}
c_\alpha &=  \lim_{R\to \infty}\Phi_\alpha\left(1-\frac{y}{R}\right) = \lim_{t\to 0^+}\Phi_\alpha(1-t) = \lim_{t\to 1^-}\Phi_\alpha(t) \\
&= \sqrt{\frac{2}{\pi}} \frac{1}{\Gamma(\alpha+\frac{1}{2})^2} \int_0^1 [u(1-u)]^{\alpha-\frac{1}{2}}du= \sqrt{\frac{2}{\pi}} \frac{B(\alpha+\frac{1}{2},\alpha+\frac{1}{2})}{\Gamma(\alpha+\frac{1}{2})^2} = \sqrt{\frac{2}{\pi}} \frac{1}{\Gamma(2\alpha+1)}>0,
\end{align*}
when $y>0$. Since $\Phi_\alpha$ extends to a continuous function on $(\infty,1]$ that is supported on $[-1,1]$ (because $\what{\varphi_\alpha}$ is), it is bounded above by some constant $d_\alpha > 0$, which gives the first part of the lemma.
\end{proof}

We then use this compute the limit we wanted.

\begin{proof}[Proof of \lemref{lem:pointwise_criterion}]
We have
\begin{align*}
R^{2\alpha+1}f_R(R+b\kappa) &= \frac{1}{\sqrt{2\pi}}R^{2\alpha+1}\int_{-\infty}^\infty \what{h_{c}^b}(x) \what{\varphi_\alpha^{R}}(R+b\kappa - x) \,dx \\
&= \frac{1}{b\sqrt{2\pi}} R^{2\alpha}\int_{-\infty}^\infty \what{h_{c}}(x/b) \what{\varphi_\alpha}\left(1 - \frac{x-b\kappa}{R}\right) dx \\
&= \frac{1}{\sqrt{2\pi}} \int_{-\infty}^\infty  \what{h_{c}}(y) R^{2\alpha} \what{\varphi_\alpha}\left(1 - b\frac{y-\kappa}{R}\right) dy \\
& = \frac{1}{\sqrt{2\pi}} \int_{\kappa}^\infty  \what{h_{c}}(y) R^{2\alpha} \what{\varphi_\alpha}\left(1 - b\frac{y-\kappa}{R}\right) dy
\end{align*}
by the change of variable $y=x/b$, where the last equality is because $\what{\varphi_\alpha}$ vanishes after $1$.

If $y > \kappa$, then $R^{2\alpha} \what{\varphi_\alpha}\left(1 - b\frac{y-\kappa}{R}\right) \leq d_\alpha (b(y- \kappa))^{2\alpha}$
for some $d_\alpha>0$ according to \lemref{lem:behaviour_near_end}. If we write $\underline \kappa = \inf_{n} \kappa_n$, then we have that $\what{h_{c}}(y) R_n^{2\alpha} \what{\varphi_\alpha}\left(1 - b\frac{y-\kappa_n}{R_n}\right)$ is bounded independently of $n$ by the integrable function
\[
d_\alpha b^{2\alpha}(y- \underline \kappa)^{2\alpha} |\what{h_{c}}(y)|
\]
on the interval $[\underline \kappa,\infty)$. We can therefore apply the dominated convergence theorem to conclude that
\begin{align*}
\lim_{n \to \infty} R_n^{2\alpha+1}f_{R_n}(R_n+b\kappa_n) &= \frac{c_\alpha b^{2\alpha}}{\sqrt{2\pi}} \int_{\kappa_0}^\infty (y - \kappa_0)^{2\alpha}\what{h_c}(y)\,dy \\
&=\frac{b^{2\alpha}}{\pi \, \Gamma(2\alpha+1)} \int_{\kappa_0}^\infty (y - \kappa_0)^{2\alpha}(c-y^2) e^{-y^2/2}\,dy ,
\end{align*}
where we used the limit from \lemref{lem:behaviour_near_end}.  That this limit depends continuously on the parameters is a consequence of the dominated convergence theorem.  If the limit is negative, then $f_{R_n}(R_n+b\kappa_n)$ is eventually negative since $R_n^{2\alpha +1}$ is positive.
\end{proof}

We now need to find good parameters where the limit in \lemref{lem:pointwise_criterion} is negative.

\begin{lem} \label{lem:negative}
Let $\alpha = 0.559$ and $c=2.3726$. Then
\[
\int_\kappa^\infty (x-\kappa)^{2\alpha} (c-x^2) e^{-x^2/2} \,dx < 0.
\]
for every $\kappa \geq \kappa_0 = 0.1814$.
\end{lem}
\begin{proof}
If $\kappa \geq \sqrt{c}$, then the result is obvious, because the integrand is non-positive. So we concentrate on the interval $[\kappa_0,\sqrt{c}]$.

Let us denote by the integral in the statement of the lemma by $I(\kappa)$. We first verify that $I(\kappa_0)<0$ using interval arithmetic in \texttt{SageMath}. Since the integrand has a singularity at $\kappa_0$, we need to estimate the integral differently near there. We write
\[
\int_{\kappa_0}^{\kappa_0+A} (x-\kappa_0)^{2\alpha} (c-x^2) e^{-x^2/2} \,dx \leq A^{2\alpha}\int_{\kappa_0}^{\kappa_0+A}(c-x^2) e^{-x^2/2} \,dx
\]
and then
\[
\int_{\kappa_0+A}^\infty (x-\kappa_0)^{2\alpha} (c-x^2) e^{-x^2/2} \,dx \leq \int_{\kappa_0+A}^B (x-\kappa_0)^{2\alpha} (c-x^2) e^{-x^2/2} \,dx
\]
as long as $A>0$, $B\geq\kappa_0+A$, and $B \geq \sqrt{c}$. With $A= 10^{-4}$ and $B=10$, these estimates provide the certified upper bound
\[
I(\kappa_0) \leq -0.0000117812526025449 < 0.
\]

We then show that $I'(\kappa) < 0$ on $[\kappa_0,0.7]$ as follows. By the change of variable $y = x - \kappa$ we get
\[
I(\kappa) = \int_0^\infty y^{2\alpha} (c-(y+\kappa)^2) e^{-(y+\kappa)^2/2} \,dy
\]
and then differentiation under the integral sign (which is justified because the derivative of the integrand is uniformly bounded by a polynomial times a Gaussian for $\kappa$ in a bounded interval) gives
\begin{align*}
I'(\kappa) &= \int_0^\infty y^{2\alpha} (y+\kappa)((y+\kappa)^2-(c+2)) e^{-(y+\kappa)^2/2} dy \\
&= \int_\kappa^\infty (x-\kappa)^{2\alpha} x (x^2 - (c+2)) e^{-x^2/2} dx\\
&= \int_\kappa^{\sqrt{c+2}} (x-\kappa)^{2\alpha} x (x^2 - (c+2)) e^{-x^2/2} dx+\int_{\sqrt{c+2}}^\infty (x-\kappa)^{2\alpha} x (x^2 - (c+2)) e^{-x^2/2}dx\\
& \leq  \int_{0.7}^{\sqrt{c+2}} (x-0.7)^{2\alpha} x (x^2 - (c+2)) e^{-x^2/2} dx +\int_{\sqrt{c+2}}^\infty (x-\kappa_0)^{2\alpha} x (x^2 - (c+2)) e^{-x^2/2}dx
\end{align*}
whenever $\kappa_0 \leq \kappa \leq 0.7$. We verify that this sum of integrals is at most \[-0.0464961225743898\] (hence negative) using interval arithmetic (again splitting the integrals near $0.7$ and $\infty$). It follows that $I$ is bounded above by $I(\kappa_0)<0$ on $[\kappa_0,0.7]$.

For any $\kappa \in [0.7,\sqrt{c}]$, we estimate
\begin{align*}
I(\kappa) &= \int_\kappa^{\sqrt{c}} (x-\kappa)^{2\alpha}(c-x^2) e^{-x^2/2} \,dx + \int_{\sqrt{c}}^\infty (x-\kappa)^{2\alpha} (c-x^2) e^{-x^2/2} \,dx \\
& \leq  \int_{0.7}^{\sqrt{c}} (x-0.7)^{2\alpha} (c-x^2) e^{-x^2/2} \,dx+\int_{\sqrt{c}}^\infty (x-\sqrt{c})^{2\alpha} (c-x^2) e^{-x^2/2} \,dx \\
&\leq -0.0907427113682867
\end{align*}
using interval arithmetic once again, which completes the proof.
\end{proof}

From the above lemma, we can deduce that the function $f$ from equation \eqref{eqn:f_sys} is non-positive from $R+b\kappa_0$ onwards provided that $R$ is large enough.

\begin{cor} \label{cor:non-positive}
Let  $\alpha = 0.559$, $b>0$, $c=2.3726$, and $\kappa_0 = 0.1814$. Then there exists some $R_0 >0$ such that $f(R+b \kappa) \leq 0$ for every $R \geq R_0$ and every $\kappa \geq \kappa_0$, where $f$ is as in equation \eqref{eqn:f_sys}.
\end{cor}

\begin{proof}
If $R>0$ and $\kappa \geq \sqrt{c}$, then
\[
f(R+b\kappa) = \frac{1}{\sqrt{2\pi}} \int_\kappa^\infty \what{h_c}(y)\what{\varphi_\alpha} \left( 1 - b \frac{y-\kappa}{R} \right) dy \leq 0
\]
because $\what{h_c}$ is non-positive after $\sqrt{c}$ and $\what{\varphi_\alpha}$ is non-negative, so their product is non-positive.

For every $\kappa  \in [\kappa_0, \sqrt{c}]$ we have that
$R^{2\alpha+1}f(R+b\kappa)$ converges to some $\mu_{\alpha,b,c,\kappa} < 0$ as $R \to \infty$ by \lemref{lem:pointwise_criterion} and \lemref{lem:negative}. Let 
\[
F(R,\kappa) := \begin{cases} R^{2\alpha+1}f(R+b\kappa) & \text{if }R \in (0,\infty) \\ \mu_{\alpha,b,c,\kappa} & \text{if }R=\infty.\end{cases}
\]
It follows form \lemref{lem:pointwise_criterion} that $F$ is continuous at every point in $\{\infty\} \times [\kappa_0, \sqrt{c}]$ and the continuity on $(0,\infty)\times [\kappa_0, \sqrt{c}]$ is a consequence of the dominated convergence theorem.

We deduce that every $\kappa \in [\kappa_0, \sqrt{c}]$, there exists some neighborhood $U$ of $\kappa$ and some $R_\kappa > 0$ such that $R^{2\alpha+1}f(R+b u) < 0$ for every $u \in U$ and every $R \geq R_\kappa$. By compactness, we can find an $R_0>0$ that works for the whole interval $[\kappa_0, \sqrt{c}]$.
\end{proof}

Armed with this result, we can finally give the proof of \thmref{thm:systole_asymp}.

\begin{proof}[Proof of \thmref{thm:systole_asymp}]
Let $\what{f}(x) = h_c(bx) \varphi_\alpha(Rx)$ as before with $\alpha = 0.559$, $b=1.0286$, $c=2.3726$, and $R>0$. Note that $c>1+b^2/4$ as required, which implies that $\what{f} \geq 0$ on $\RR \cup i \left[\frac12,\frac12\right]$. By \corref{cor:non-positive}, there exists some $R_0 > 0$ such that if $R\geq R_0$, then $f(x) \leq 0$ whenever $x \geq R + b\kappa_0$, where $\kappa_0 = 0.1814$ (recall that $f$ itself depends on $R$, which is why we need to consider the parameters $\kappa \geq \kappa_0$ to cover every $x \geq R + b\kappa_0$). Therefore, $R+b\kappa_0$ provides a bound on $\sys(M)$ as long as 
\[
\what{f}(i/2) \geq 2(g-1) \int_0^\infty \what{f}(x) x \tanh(\pi x) \,dx.
\]
We will thus estimate both sides and choose $R$ so that this inequality is satisfied.

For the left-hand side, we have
\[
\what{f}(i/2) = h_c(bi/2)\varphi_\alpha(iR/2)
\]
where
\[
\varphi_\alpha(iR/2) = \frac{J_\alpha(iR/4)^2}{(R/2)^{2\alpha}} \sim \frac{2^{2\alpha+1}}{\pi} \frac{e^{R/2}}{R^{2\alpha+1}}
\]
as $R \to \infty$ by the asymptotic \eqref{eq:Bessel_assymptotic_imag}.

As for the integral term, we consider
\[
 \int_0^\infty h_c(bx)  x \tanh(\pi x) R^{2\alpha+1} \varphi_\alpha(Rx)\,dx
\]
and recall from \secref{sec:Bessel} that for every $x>0$ we have
\[
\left| R^{2\alpha+1} \varphi_\alpha(Rx) - \frac{4}{\pi x^{2\alpha+1}} \cos^2 \left( \frac{Rx}{2} - \frac{(2\alpha+1)\pi}{4} \right) \right|  \to 0
\] 
as $R \to \infty$ and that $y \mapsto \sqrt{y} J_\alpha(y)$ is bounded on $(0,\infty)$. Since $(Rx)^{2\alpha+1} \varphi_\alpha(Rx) = Rx J_\alpha(Rx/2)^2$, there exists a constant $C>0$ such that $R^{2\alpha+1} \varphi_\alpha(Rx) \leq  C x^{-(2\alpha+1)}$ for every $x>0$ and every $R>0$. Observe that $h_c(bx)x\tanh(\pi x) x^{-(2\alpha+1)}$  is integrable on $(0,\infty)$ as long as $\alpha<1$. Since this is satisfied for our chosen parameter $\alpha = 0.559$, we can apply the dominated convergence theorem to obtain
\begin{align*}
& \lim_{R \to \infty}  R^{2\alpha+1}\int_0^\infty h_c(bx)  x \tanh(\pi x)  \varphi_\alpha(Rx)\,dx \\
& = \frac{4}{\pi} \lim_{R \to \infty} \int_0^\infty \frac{h_c(bx)\tanh(\pi x)}{x^{2\alpha}} \cos^2 \left( \frac{Rx}{2} - \frac{(2\alpha+1)\pi}{4} \right) dx.
\end{align*}
To compute the integral inside the limit, we write the squared cosine as
\[
\frac12\left( 1+ \cos(Rx)\cos\left(\frac{(2\alpha+1)\pi}{2}\right) +  \sin(Rx)\sin\left(\frac{(2\alpha+1)\pi}{2}\right)  \right).
\]
By the Riemann--Lebesgue lemma, the integrals of the terms with $\cos(Rx)$ or $\sin(Rx)$ tend to zero as $R \to \infty$, so that
\[
\frac{4}{\pi} \lim_{R \to \infty} \int_0^\infty \frac{h_c(bx)\tanh(\pi x)}{x^{2\alpha}} \cos^2 \left( \frac{Rx}{2} - \frac{(2\alpha+1)\pi}{4} \right) dx = \frac{2}{\pi} \int_0^\infty \frac{h_c(bx) \tanh(\pi x)}{x^{2\alpha}}dx.
\]
We thus have
\[
\frac{\what{f}(i/2)}{\int_0^\infty \what{f}(x) x \tanh(\pi x) \,dx} \sim \frac{ 4^{\alpha} h_c(bi/2)}{\int_0^\infty \frac{h_c(bx) \tanh(\pi x)}{x^{2\alpha}}dx} e^{R/2}
\]
as $R \to \infty$.

For any $\rho>1$ and $g\geq 2$, if we choose $R$ such that the right-hand side equal is equal to $2(g-1)\rho$, then the left-hand side will be larger than $2(g-1)$ provided that $R$ is large enough so that the asymptotic is sufficiently precise. We thus take
\[
R = 2 \log(g-1)+2\log\left( \rho \frac{2 ^{1-2\alpha}}{h_c(bi/2)} \int_0^\infty \frac{h_c(bx) \tanh(\pi x)}{x^{2\alpha}}dx  \right)   ,
\]
which tends to infinity as $g \to \infty$. The hypotheses of \thmref{thm:systole} will then  satisfied if $g$ is large enough, and we can further assume that $R \geq R_0$, so that the resulting bound on the systole is
\begin{align*}
\sys(M) &\leq R + b \kappa_0 \\
& = 2 \log(g-1) + 2\log(\rho) +b \kappa_0 + 2 \log \left( \frac{2 ^{1-2\alpha}}{h_c(bi/2)} \int_0^\infty \frac{h_c(bx) \tanh(\pi x)}{x^{2\alpha}}dx  \right)
\end{align*}
according to \corref{cor:non-positive}. Since $\rho>1$ was arbitrary, the additive constant can be taken as close as we wish to 
\[
b \kappa_0 + 2 \log \left( \frac{2 ^{1-2\alpha}}{h_c(bi/2)} \int_0^\infty \frac{h_c(bx) \tanh(\pi x)}{x^{2\alpha}}dx  \right).
\]

To estimate the integral rigorously, we restrict to a compact interval $[A,B] \subset (0,\infty)$ and  estimate the remaining parts by
\begin{align*}
\int_{(0,\infty) \setminus [A,B]} \frac{h_c(bx) \tanh(\pi x)}{x^{2\alpha}}dx &\leq \pi  \int_0^A (c - 1 + b^2 A^2) x^{1-2\alpha} \,dx + \int_B^\infty \frac{h_c(bx)}{x^{2\alpha}} dx \\
& \leq  \pi  (c - 1 + b^2 A^2)\frac{A^{2-2\alpha}}{2-2\alpha}+ \frac{h_c(bB)}{(2\alpha-1)B^{2\alpha-1}},
\end{align*}
as long as $B$ is large enough so that $h_c$ is decreasing on $[bB,\infty)$. Since \[h_c'(x) = x((3-c)- x^2)e^{-x^2/2}\] is negative when $x > \sqrt{3-c}$ and since our parameters satisfy $\sqrt{3-c} < 1 < b$, any $B\geq 1$ works.

With $A = 10^{-6}$ and $B=10$, interval arithmetic in \texttt{SageMath} certifies that
\[
b \kappa_0 + 2 \log \left( \frac{2 ^{1-2\alpha}}{h_c(bi/2)} \int_0^\infty \frac{h_c(bx) \tanh(\pi x)}{x^{2\alpha}}dx  \right) \leq 2.40896511079437 < 2.409
\]
at the given parameters, as required.
\end{proof}

\section{Kissing numbers}

\subsection{The criterion} \label{subsec:kissing_criterion}
The kissing number of a hyperbolic surface $M$ is defined as its number of oriented closed geodesics of minimal length. In the literature it is common to consider the quantity $\frac12\kiss(M)$, which counts the number of unoriented systoles in $M$. Our criterion for bounding kissing numbers is the same as the one we used in \cite{KissingManifolds} for hyperbolic manifolds in any dimension. We repeat the proof in the surface case for convenience.

\begin{thm} \label{thm:kissing}
Let $s>0$ and suppose that $f$ is an admissible function such that
\begin{itemize}
\item $\what{f}(\xi) \geq 0$ for every $\xi \in \RR \cup i\left[-\frac12,\frac12\right]$;
\item $f(x) \leq 0 $ whenever $x \geq s$ and $f(s) < 0$.
\end{itemize}
Then for every closed hyperbolic surface $M$ of genus $g\geq 2$ such that $\sys(M)=s$ we have
\[
\kiss(M) \leq 4\sqrt{2\pi}(g-1)\frac{\sinh(s/2)}{s| f(s)|} \int_0^\infty \what{f}(x) x\tanh(\pi x) \,dx.
\]
\end{thm}
\begin{proof}
We have
\begin{align*}
0 & \leq \sum_{j=0}^\infty \what{f}\left(\sqrt{\lambda_j(M) - \frac14}\right) \\
& = 2(g-1)\int_0^\infty \what{f}(x) x\tanh(\pi x) \,dx + \frac{1}{\sqrt{2\pi}}\sum_{\gamma \in \calC(M)} \frac{\Lambda(\gamma)}{2 \sinh(\ell(\gamma)/2)} f(\ell(\gamma)) \\
&\leq 2(g-1)\int_0^\infty \what{f}(x) x\tanh(\pi x) \,dx +  \frac{\kiss(M)}{\sqrt{2\pi}} \frac{\sys(M)}{2\sinh(\sys(M)/2)} f(\sys(M))\\
\end{align*}
since the contribution of geodesics longer than the systole is non-positive. Rearranging gives the desired result.
\end{proof}

\begin{rem}
We could take the value of $\what{f}(\sqrt{-1/4})=\what{f}(i/2)$ into account to get an a priori better bound, but in practice we have found that the optimal functions seem to satisfy $\what{f}(i/2)=0$. We therefore enforce this condition in our program to speed up convergence.
\end{rem}

\begin{defn}
For $s>0$, we define $K(s)$ as the infimum of 
\[
4\sqrt{2\pi} \frac{\sinh(s/2)}{s| f(s)|} \int_0^\infty \what{f}(x) x\tanh(\pi x) \,dx
\]
over the functions $f$ that satisfy the hypotheses of \thmref{thm:kissing}. 
\end{defn}

\thmref{thm:kissing} can then be restated as saying that 
\[
\kiss(M) \leq K(\sys(M)) (g-1).
\]
The following monotonicity result will greatly simplify our task of finding a global bound (independent of the systole) in each genus.

\begin{lem}  \label{lem:non-decreasing}
The function $K(s)$ is non-decreasing for $s\in [6,\infty)$.
\end{lem}
\begin{proof}
Let $f$ be a function that satisfies the hypotheses of \thmref{thm:kissing} at some $s_2> 6$ and let $s_1 \in [6, s_2)$. We then consider the function $\phi(x) = f\left(\frac{s_2}{s_1} x \right)$ with $\what{\phi}(x) = \frac{s_1}{s_2} \what{f}\left( \frac{s_1}{s_2} x\right)$. By the hypotheses on $f$ we have that $\phi(x) = f\left(\frac{s_2}{s_1} x \right) \leq 0$ whenever $x\geq s_1$, because that implies $\frac{s_2}{s_1} x \geq s_2$. Moreover, we have $\phi(s_1) = f(s_2)<0$. Secondly, if $\xi \in \RR \cup i \left[-\frac12,\frac12\right]$, then $\frac{s_1}{s_2} \xi \in \RR \cup i \left[-\frac12 \frac{s_1}{s_2},\frac12 \frac{s_1}{s_2}\right] \subset \RR \cup i \left[-\frac12,\frac12\right]$ so that $\what{\phi}(\xi) = \frac{s_1}{s_2} \what{f}\left( \frac{s_1}{s_2} \xi\right) \geq 0$.

It remains to estimate the resulting bounds on $K$. We compute
\begin{align*}
\frac{\sinh(s_1/2)}{s_1}\int_0^\infty \what{\phi}(x) x \tanh(\pi x) \,dx &= \frac{\sinh(s_1/2)}{s_1} \frac{s_1}{s_2} \int_0^\infty \what{f}\left( \frac{s_1}{s_2} x\right) x \tanh(\pi x) \,dx \\
&= \frac{\sinh(s_1/2) s_2}{s_1^2} \int_0^\infty \what{f}(y) y \tanh\left( \frac{s_2}{s_1} \pi y  \right)dy \\
&\leq \frac{\sinh(s_1/2) s_2^2}{s_1^3} \int_0^\infty \what{f}(y) y \tanh\left( \pi y  \right)dy.
\end{align*}

We then claim that
\[
\frac{\sinh(s_1/2) s_2^2}{s_1^3} \leq \frac{\sinh(s_2/2)}{s_2}.
\]
Indeed, this is equivalent to saying that $\sinh(x/2)/x^3$ increases from $s_1$ to $s_2$ and elementary computations show that the derivative of this function is non-negative provided that $x$ is at least $6 \tanh(x/2)$, which is certainly true if $x \geq 6$.

We conclude that
\[
\frac{\sinh(s_1/2)}{s_1 |\phi(s_1)|}\int_0^\infty \what{\phi}(x) x \tanh(\pi x) \,dx \leq \frac{\sinh(s_2/2)}{s_2|f(s_2)|} \int_0^\infty \what{f}(y) y \tanh\left( \pi y  \right)dy
\]
and hence that $K(s_1) \leq K(s_2)$ upon taking the infimum over $f$.
\end{proof}

When the systole is sufficiently small, any two shortest closed geodesics are disjoint, which implies a bound on the kissing number for topological reasons. More precisely, by the collar lemma  \cite[Theorem 4.1.6]{Buser}, if $\sys(M) \leq 2\arcsinh(1)$, then $\kiss(M) \leq 6(g - 1)$ (and this is optimal).

In order to obtain a global upper bound on the kissing number in a given genus $g$ from \thmref{thm:kissing}, it remains to prove an upper bound for $K(s)$ when $s \in [2\arcsinh(1), 6]$ and at $s=\mathrm{sys\_bound}(g)$ where $\mathrm{sys\_bound}(g)$ is the bound on the systole in genus $g$ coming from the previous section. This is if $\mathrm{sys\_bound}(g)>6$, which happens as soon as $g\geq 7$. In lower genus, we only have to bound $K(s)$ for $s$ in $[2\arcsinh(1),\mathrm{sys\_bound}(g)]$.

\subsection{Low genus}

To deal with the interval $[2\arcsinh(1), 6]$, we partition it into smaller subintervals, and then optimize to find a single function $f$ for each subinterval. One subtlety is that to get an upper bound on $K(s)$ when $s$ belongs to an interval $[a,b]$, we need an upper bound on 
\[
\frac{\sinh(s/2)}{-sf(s)}
\]
for $s$ in the same interval. Once we find a candidate upper bound numerically, we appro\-xi\-mate $\sinh(s/2)$ and $e^{-s^2/2}$ by Taylor polynomials from the correct direction and then use Sturm's theorem to certify that the bound is valid (recall that $f(s) = p(s^2)e^{-s^2/2}$ for some polynomial $p$).

In other words, if we know that $f(s)<0$ for every $s \in [a,b]$, then
\[
 \frac{\sinh(s/2)}{-sf(s)} \leq B  \quad \Longleftrightarrow \quad \sinh(s/2) + Bs \,p(s^2)e^{-s^2/2} \leq 0.
\]
If $S_n(x) \geq \sinh(x/2)$ and $E_n(x) \leq e^{-x^2/2}$ are polynomial approximations, then it suffices to check that
\[
S_n(x) + B x \,p(x^2) E_n(x) < 0,
\]
on $[a,b]$ (we are assuming that $p(x^2)<0$ there, hence the reversed inequality for $E_n$). We can perform this verification by checking that $S_n(x) + B x \,p(x^2) E_n(x)$ is negative at one point (using interval arithmetic) and has no zeros in $[a,b]$ (using Sturm's theorem). For $E_n$, we simply take the odd degree Taylor polynomials of the exponential evaluated at $-x^2/2$ (the resulting series is alternating). For $S_n$, we use Taylor's theorem with the Lagrange form of the remainder to get
\[
\sinh(x/2) = \sum_{j=0}^{n-1} \frac{(x/2)^{2j+1}}{(2j+1)!} + \cosh(\xi_x) \frac{(x/2)^{2n+1}}{(2n+1)!}
\]
for some $\xi_x \in [0, x]$ and thus
\[
\sinh(x/2) \leq \sum_{j=0}^{n-1} \frac{(x/2)^{2j+1}}{(2j+1)!} + \cosh(b) \frac{(b/2)^{2n+1}}{(2n+1)!}
\]
for every $x \in [a,b] \subset [0,b]$. We define $S_n(x)$ by the above formula except that we replace the last term involving $b$ by a rational approximation from above, so that the computer can apply Sturm's theorem reliably.

The upper bounds on $\frac12 \kiss$ that result from the above strategy are shown in Table \ref{table:kiss} for genus $2$ to $20$ and verified in the ancillary file \texttt{verify\_kissing.ipynb}. They improve upon the previous best bounds in every genus except $g=2$, where the optimal bound is $12$ for the number of unoriented systoles \cite{SchmutzKissing}. We remark that these upper bounds depend in a very sensitive way on the upper bounds on the systole from Table \ref{table:sys}, which are not as small as possible because we took precautions to make sure that they were rigorous. Consequently, the upper bounds in Table \ref{table:kiss} could be decreased with more precision (especially those towards the end of the table).

{\footnotesize
\begin{table}[htp]
\centering
\caption{Bounds on the maximum of $\frac12\kiss$. The dagger indicates where the linear programming bound fails to beat the previous best upper bound.}  \label{table:kiss}
\begin{tabular}{|l|l|l|l|}
\hline
genus & lower bound & LP bound & previous upper bound  \\
\hline
2 & 12 \cite{Jenni} & $14^\dagger$ & 12 \cite{SchmutzKissing}\\
3 & 24 \cite{SchmutzMaxima}& 34 & 126 \cite{MRT}\\
4 & 36 \cite{SchmutzMaxima}& 62  & 244 \cite{KissingManifolds} \\
5 & 48 \cite{SchmutzMaxima}& 97 & 383 \cite{KissingManifolds} \\
6 &  39 \cite{Hamenstadt}   & 138 & 547 \cite{KissingManifolds}\\
7 & 126 \cite{VogelerThesis} & 185 & 736 \cite{KissingManifolds}\\
8 &       & 240 & 950 \cite{KissingManifolds}\\
9 &  70 \cite{SchmutzMaxima}     & 299 & 1186 \cite{KissingManifolds} \\
10 &     & 364 & 1446 \cite{KissingManifolds}\\
11 &  120 \cite{SchmutzMaxima}   & 434 & 1728 \cite{KissingManifolds}\\
12 &     & 510 & 2032 \cite{KissingManifolds}\\
13 & 144 \cite{FBRafi} & 591 & 2358 \cite{KissingManifolds}\\
14 & 364 \cite{VogelerThesis} & 677 & 2706 \cite{KissingManifolds}\\
15 & 168 \cite{FBRafi} & 771 & 3074 \cite{KissingManifolds}\\
16 & 180 \cite{FBRafi} & 868 & 3464 \cite{KissingManifolds}\\
17 & 336 \cite{VogelerThesis} & 970 & 3874 \cite{KissingManifolds}\\
18 &  204 \cite{FBRafi} & 1083 & 4305 \cite{KissingManifolds}\\
19 & 216 \cite{FBRafi} & 1209 & 4756 \cite{KissingManifolds}\\
20 & 228 \cite{FBRafi} & 1333 & 5227 \cite{KissingManifolds}\\
\hline
\end{tabular}
\end{table}
}

\subsection{Asymptotics}

To prove an asymptotic upper bound the kissing number, we start with a proposition that bounds this quantity in terms of the systole.

\begin{prop}  \label{prop:kiss_sys}
There exist some $s_0>0$ such that
\[
K(s) <  18.355 \cdot \frac{2\sinh(s/2)}{s}
\]
for every $s \geq s_0$. In particular, every closed hyperbolic surface $M$ of genus $g\geq 2$ and systole $\sys(M)\geq s_0$ satisfies
\[
\kiss(M) < 18.355 \cdot \frac{2\sinh(\sys(M)/2)}{\sys(M)} (g-1).
\]
\end{prop}

\begin{rem}
This improves upon \cite[Remark 4.4]{KissingManifolds}, where we obtained the same inequality but with the constant $63.71$ instead of $18.355$.
\end{rem}

\begin{proof}
Recall that
\[
K(s)=  2\sqrt{2\pi}\frac{2\sinh(s/2)}{s} \inf_f\left\{\frac{1}{-f(s)}\int_0^\infty \what{f}(x) x \tanh(x)\,dx \right\}
\]
where the infimum is over admissible functions $f$ such that $f(x)\leq 0$ if $x \geq s$, $f(s)<0$, and $\what{f}(\xi) \geq 0$ if $\xi \in \RR \cup i [-\frac12,\frac12]$.

To prove the desired inequality, we use the same functions as for the asymptotic systole bound but choose the parameters differently. That is, we take $f$ such that
\[
\what{f}(x) = (c-1+b^2 x^2)e^{-b^2x^2/2} \varphi_\alpha(Rx) = (c-1+b^2 x^2)e^{-b^2x^2/2}  \frac{J_\alpha(Rx/2)^2}{(Rx)^{2\alpha}}
\]
for some parameters $\alpha$, $b$, $c$ and $R$ but now set $c = 1+b^2/4$ so that $\what{f}(i/2)=0$.

Recall from \lemref{lem:pointwise_criterion} that for any $\kappa_0 \in \RR$, we have that $R^{2\alpha+1}f(R+b\kappa_0)$ tends to
\[
-L_1 :=  \frac{b^{2\alpha}}{\pi\,\Gamma(2\alpha+1)} \int_{\kappa_0}^\infty (x-\kappa_0)^{2\alpha}(c-x^2)e^{-x^2/2}\,dx,
\]
as $R \to \infty$. We also saw in the proof of \thmref{thm:systole_asymp} that
\begin{align*}
L_2 := \lim_{R\to\infty} R^{2\alpha+1} \int_0^\infty \what{f}(x) x \tanh(\pi x) \,dx &= \frac{2}{\pi}\int_0^\infty \frac{\tanh(\pi x)}{x^{2\alpha}}(c-1+(bx)^2)e^{-(bx)^2/2}\,dx \\
&= \frac{2 b^2}{\pi}\int_0^\infty \frac{\tanh(\pi x)}{x^{2\alpha}}(x^2+1/4)e^{-b^2x^2/2}\,dx.
\end{align*}

Our goal is thus to minimize the ratio $L_2/L_1$ over the parameters such that $-L_1 < 0$ and such that $f$ is non-positive from $s=R+b\kappa_0$ onwards. At $\alpha = 0.592$, $b = 0.981$ and $\kappa_0 = 0.061$, we obtain $2\sqrt{2\pi} L_2/L_1 < 18.355$, which implies that
\[
2\sqrt{2\pi}\frac{1}{-f(R+b\kappa_0)}\int_0^\infty \what{f}(x) x \tanh(x)\,dx  < 18.355 
\]
if $R$ is large enough. To prove that $2\sqrt{2\pi} L_2/L_1 < 18.355$, we observe that
\[
\int_0^\infty \frac{\tanh(\pi x)}{x^{2\alpha}}(x^2+1/4)e^{-b^2x^2/2}\,dx
\]
is bounded above by 
\[
\pi\left(A^2+\frac14\right) \frac{A^{2(1-\alpha)}}{2(1-\alpha)} + \int_A^B \frac{\tanh(\pi x)}{x^{2\alpha}}(x^2+1/4)e^{-b^2x^2/2}\,dx + \frac{1}{B^{2\alpha}}\int_B^\infty x^3 e^{-b^2x^2/2}\,dx
\]
whenever $0<A < 2 < B$ and estimate each term using interval arithmetic with $A=1/10^4$ and $B=10^4$, noting that the last term can be rewritten as
\[
\frac{1}{B^{2\alpha}}\int_B^\infty x^3 e^{-b^2x^2/2}\,dx = \frac{1}{b^4 B^{2\alpha}} \int_{bB}^\infty y^3e^{-y^2/2}\,dy =  \frac{(b^2 B^2 +2)}{b^4 B^{2\alpha}} e^{-b^2B^2/2}.
\]
For $L_1$, we need a lower bound. We have
\[
\int_0^\infty \frac{1-\cos(x)}{x^{2\alpha+1}}dx \geq \int_A^B\frac{1-\cos(x)}{x^{2\alpha+1}}dx
\]
since the integrand is non-negative. We then observe that
\[
-\int_{\kappa_0}^\infty (x-\kappa_0)^{2\alpha}(c-x^2)e^{-x^2/2}\,dx \geq -\int_{\kappa_0}^B (x-\kappa_0)^{2\alpha}(c-x^2)e^{-x^2/2}\,dx,
\]
since $B > \sqrt{c}$. We split this last integral at $\kappa_0+A$ and estimate
\begin{align*}
-\int_{\kappa_0}^{\kappa_0+A} (x-\kappa_0)^{2\alpha}(c-x^2)e^{-x^2/2}\,dx &\geq -A^{2\alpha} \int_{\kappa_0}^{\kappa_0+A} (c-x^2)e^{-x^2/2}\,dx
\end{align*}
Putting all the estimates together, we obtain the certified upper bound
\[
2\sqrt{2\pi} \frac{L_2}{L_1} \leq 18.3545723139258 < 18.355.
\]

The last thing to check is that when $R$ is large enough, we have $f(x)\leq 0$ for every $x \geq s = R+b\kappa_0$. The proof of this fact is similar as for the systole bound, and is deferred to the next lemma. 

The resulting upper bound on $\kiss(M)$ when $\sys(M)\geq s_0$ then follows from \thmref{thm:kissing}.
\end{proof}

We now prove a small lemma which verifies that the function $f$ used above satisfies the hypotheses of \thmref{thm:kissing} for $s=R+b\kappa_0$ whenever $R$ is large enough. This is similar to \lemref{lem:negative} and \corref{cor:non-positive}.

\begin{lem}
Let $f$ be as in \propref{prop:kiss_sys} with $\alpha = 0.592$, $b = 0.981$, $c=1+b^2/4$, and $\kappa_0 = 0.061$. Then there exists some $R_0>0$ such that $f(x) \leq 0$ for every $x \geq R+b\kappa_0$ and every $R \geq R_0$. 
\end{lem}
\begin{proof}
We begin by proving the pointwise result that for every $\kappa \geq \kappa_0$, we have that $f(R+b\kappa) \leq 0$ if $R$ is sufficiently large. Since $f$ is the convolution of a non-negative function supported in $[-R,R]$ and a function which is non-positive outside $[-b\sqrt{c},b\sqrt{c}]$, it is obviously non-positive at points $x$ with $|x|\geq R+b\sqrt{c}$. In other words, the result is obvious (and holds for every $R>0$) if $\kappa \geq \sqrt{c}$. 

For $\kappa \in [\kappa_0,\sqrt{c}]$, we use the fact that $R^{2\alpha+1}f(R+b\kappa)$ converges to a positive multiple of
\[
I(\kappa)=\int_{\kappa}^\infty (x-\kappa)^{2\alpha}(c-x^2)e^{-x^2/2}\,dx
\]
as $R \to \infty$, so it suffices to check that $I(\kappa) <0$ for every $\kappa \in [\kappa_0,\sqrt{c}]$. This is similar to the statement of \lemref{lem:negative} but is easier to prove because $I(\kappa_0)$ is not close to $0$, so coarse bounds suffice. We write
\begin{align*}
I(\kappa) &= \int_{\kappa}^{\sqrt{c}} (x-\kappa)^{2\alpha}(c-x^2)e^{-x^2/2}\,dx + \int_{\sqrt{c}}^\infty (x-\kappa)^{2\alpha}(c-x^2)e^{-x^2/2}\,dx \\
& \leq \int_{\kappa_0}^{\sqrt{c}} (x-\kappa_0)^{2\alpha}(c-x^2)e^{-x^2/2}\,dx + \int_{\sqrt{c}}^\infty (x-\sqrt{c})^{2\alpha}(c-x^2)e^{-x^2/2}\,dx
\end{align*}
then split the first integral at $\kappa_0+1/100$ and the second one at $100$ to get that
\[
I(\kappa) \leq -0.276593809735452<0
\]
for every $\kappa \in [\kappa_0,\sqrt{c}]$.

From this pointwise result and the same continuity and compactness argument as in \corref{cor:non-positive}, we obtain that there exists some $R_0>0$ such that $f(R+b\kappa) \leq 0$ for every $\kappa \geq \kappa_0$ and every $R \geq R_0$. In other words, we have that $f(x)\leq 0$ for every $x \geq R+b\kappa$ and every $R \geq R_0$.
\end{proof}

We then combine the previous proposition with our asymptotic bound for the systole to obtain the following bound on kissing numbers that only depends on the genus.

\begin{thm}  \label{thm:kiss_asymp}
There exists some $g_0 \geq 2$ such that every closed hyperbolic surface $M$ of genus $g\geq g_0$ satisfies 
\begin{align*}
\kiss(M) < \frac{30.608 \cdot g^2}{\log(g)+1.2045}.
\end{align*}
\end{thm}

\begin{rem}
For some towers of principal congruence covers of arithmetic surfaces, the kissing number grows at least like $g^{4/3 - \eps}$ for any $\eps>0$ \cite{Schmutz43}. This was recently generalized to higher dimensional hyperbolic manifolds in \cite{KissingHigher} with $g$ replaced by the volume and an exponent that depends on the dimension.
\end{rem}

\begin{proof}[Proof of \thmref{thm:kiss_asymp}]
Recall that \thmref{thm:systole_asymp} states that \[\sys(M) \leq 2\log(g) + 2.409\]
if $g$ is large enough. Let $s_0$ be as in \propref{prop:kiss_sys}. If $g$ is sufficiently large, then $2\log(g) + 2.409 \geq s_0$ and we get
\begin{align*}
K(2\log(g) + 2.409) &\leq  18.355\, \frac{2\sinh((2\log(g) + 2.409)/2)}{2\log(g) + 2.409} \\
&< 9.1775\, \frac{e^{1.2045} g}{\log(g) + 1.2045} \\
&<  \frac{30.608 \cdot g}{\log(g) + 1.2045}
\end{align*}
where we used the fact that $2\sinh(x/2) < e^{x/2}$ for every real number $x$.

By \thmref{thm:kissing}, we have
\[
\kiss(M) \leq K(\sys(M)) (g-1).
\]
Furthermore, \lemref{lem:non-decreasing} says that $K$ is non-decreasing on $[6,\infty)$. We thus get that
\[
\kiss(M) \leq K(2\log(g) + 2.409) (g-1) < \frac{30.608 \cdot g^2}{\log(g) + 1.2045}.
\]
provided $\sys(M)\geq 6$ and $2\log(g) + 2.409 \geq \max\{ 6 , s_0 \}$, which holds whenever $g$ is large enough.

When the systole is at most $2\arcsinh(1)$, we noted previously that $\kiss(M)\leq 6(g-1)$ by the collar lemma, and this quantity is smaller than the stated bound when $g$ is large enough (in fact, this is true for all $g\geq 2$). 

The only interval left to cover is $[2\arcsinh(1),6]$. By the calculations used to produce \tableref{table:kiss}, the function $K$ is bounded on that interval. One can also prove this using a single function $f$ defined by
\[
\what{f}(x) = (x^2+1/4)e^{-x^2/2} \varphi_\alpha(Rx)
\]
with $R = 2\arcsinh(1)-\sqrt{5/4}>0$ and any $\alpha\in(0,1)$, because it satisfies all the hypotheses of \thmref{thm:kissing} for every $s \geq 2\arcsinh(1)$. The resulting linear upper bound on $\kiss(M)$ when $\sys(M) \in [2\arcsinh(1),6]$ is eventually smaller than $\frac{30.608 \cdot g^2}{\log(g)+1.2045}$ when $g$ is large enough.
\end{proof}

\section{First eigenvalue}

\subsection{The criterion} 
The criterion for bounding the first positive eigenvalue $\lambda_1(M)$ of the Laplacian on $M$ goes as follows.

\begin{thm} \label{thm:lambda}
Let $g \geq 2$. Suppose that $f$ is a non-constant admissible function for which there exists an $L>0$ such that
\begin{itemize}
\item $f(x) \geq 0$ for all $x\in \RR$;
\item $\what{f}\left(\sqrt{\lambda-\frac14}\right) \leq 0 $ whenever $\lambda \geq L$;
\item $\what{f}(i/2) \leq 2(g-1) \int_0^\infty \what{f}(x) x\tanh(\pi x)\,dx$;
\end{itemize}
Then $\lambda_1(M) \leq L$ for every hyperbolic surface $M$ of genus $g$.
\end{thm}

\begin{proof}
Suppose that $M$ is a hyperbolic surface with $\lambda_1(M) > L$. By continuity, the same inequality holds for every surface $N$ in some neighborhood $U$ of $M$ in moduli space. Let $r_j(N)\in \CC$ be such that $r_j(N)^2 =\lambda_j(N)-\frac14$. The Selberg trace formula yields
\begin{align*}
\sum_{j =0}^\infty \what{f}(r_j(N)) & \leq \what{f}(i/2) \\
&  \leq  2(g-1) \int_0^\infty \what{f}(x) x\tanh(\pi x)\,dx \\
 &  \leq 2(g-1) \int_0^\infty \what{f}(x) x\tanh(\pi x)\,dx + \frac{1}{\sqrt{2\pi}}\sum_{\gamma \in \calC(N)} \frac{\Lambda(\gamma)f(\ell(\gamma))}{2 \sinh(\ell(\gamma)/2)}  \\
& = \sum_{j =0}^\infty \what{f}(r_j(N))
\end{align*}
from which we conclude that $\what{f}(r_j(N))=0$ for every $j\geq 1$ and every $N\in U$. As in the proof of \thmref{thm:systole}, this leads to a contradiction since the zero set of $\what{f}$ is discrete and there are no non-trivial isospectral deformations of a hyperbolic surface. We conclude that $\lambda_1(M) \leq L$ for every $M$.
\end{proof}

\begin{rem}
If the inequality in the third bullet point is strict, then the conclusion can be strengthened to a strict inequality and this is easier to prove. 
\end{rem}

\subsection{Low genus}

The upper bounds on $\lambda_1(M)$ resulting from \thmref{thm:lambda} and numerical optimization are presented in \tableref{table:lambda} for $2\leq g \leq 20$ and their verification is done in the file \texttt{verify\_lambda.ipynb}. Our bounds are smaller than the previous best upper bounds in every genus except $2$, $3$, $4$, and $6$ where bounds from \cite{bootstrap}, \cite{bootstrap2} or \cite{YangYau} are better. 

{\footnotesize
\begin{table}[htp] 
\centering
\caption{Bounds on the supremum of $\lambda_1$. The daggers indicate where the linear programming bounds fail to beat the previous best upper bound.
} \label{table:lambda}
\begin{tabular}{|l|l|l|l|}
\hline
genus & lower bound & LP bound  & previous upper bound  \\
\hline
2 &  3.838887 \cite{StrohmaierUski} & $4.625307^\dagger$ & 3.838898 \cite{bootstrap,bootstrap2}\\
3 & 2.6779 \cite{Cook} & $2.816427^\dagger$ & 2.678483 \cite{bootstrap,bootstrap2}\\
4 & 1.91556 \cite{Cook} & $2.173806^\dagger$ & 2.000000  \cite{YangYau}\\
5 &   0.728167 (\S\ref{subsec:gaps})   & 1.836766 & 1.852651 \cite{bootstrap}\\
6 &   0.486360 (\S\ref{subsec:gaps})      & $1.625596^\dagger$ & 1.600000 \cite{YangYau}\\
7 &   1.23 \cite{Lee}     & 1.480008 & 1.513268 \cite{bootstrap}\\
8 &  0.25 \cite{BookerStrombergsson}+\cite{BBD} & 1.372804 & 1.406905 \cite{bootstrap}\\
9 &   0.25 \cite{BookerStrombergsson}+\cite{BBD}       & 1.289024 & 1.323482 \cite{bootstrap}\\
10 &   0.25 \cite{BookerStrombergsson}+\cite{BBD}      & 1.222189 & 1.256022 \cite{bootstrap}\\
11 &   0.25 \cite{BookerStrombergsson}+\cite{BBD}      & 1.168169  & 1.200153 \cite{bootstrap}\\
12 &   0.25 \cite{BookerStrombergsson}+\cite{BBD}      & 1.122327 & 1.152986 \cite{bootstrap}\\
13 &   0.25 \cite{BookerStrombergsson}+\cite{BBD}      & 1.083260 & 1.112535  \cite{bootstrap}\\
14 &   0.287470 (\S\ref{subsec:gaps})    & 1.049217 & 1.077385 \cite{bootstrap}\\
15 &   0.25 \cite{BookerStrombergsson}+\cite{BBD}   & 1.018005 & 1.046501 \cite{bootstrap}\\
16 &    0.25 \cite{BookerStrombergsson}+\cite{BBD}     & 0.991735 & 1.019105 \cite{bootstrap}\\
17 &   0.403200 (\S\ref{subsec:gaps})    & 0.968260 & 0.994601 \cite{bootstrap}\\
18 &     0.25 \cite{BookerStrombergsson}+\cite{BBD}    & 0.947180 & 0.972525 \cite{bootstrap}\\
19 &      0.25 \cite{BookerStrombergsson}+\cite{BBD}   & 0.928091 & 0.952510 \cite{bootstrap}\\
20 &     0.25 \cite{BookerStrombergsson}+\cite{BBD}    & 0.911390 & 0.934260 \cite{bootstrap}\\
\hline
\end{tabular}
\end{table}
}

Some comments on the examples we used for the lower bounds are in order:
\begin{itemize}
\item Contrary to the other invariants considered in this paper, it is not known if the supremum of $\lambda_1$ is attained. For instance, we do not know if the entries equal to $1/4$ in the table are attained (see below).

\item In genus $2$ and $3$, the upper bounds from \cite{bootstrap,bootstrap2} are tantalizingly close to the value of $\lambda_1$ at the Bolza surface and the Klein quartic approximated numerically in \cite{StrohmaierUski} and \cite{Cook} respectively, so these surfaces are the conjectured maximi\-zers in these genera. The authors of \cite{bootstrap} reproduced Cook's numerical calculations with more precision, arriving at the value $2.6779$ instead of $2.6767$ for the Klein quartic. The surface in genus $4$ is Bring's curve. Note that the values in genus $3$ and $4$ are based on finite element methods and are not rigorous. The value in genus $2$ is obtained using the trace formula and can be made rigorous according to Strohmaier and Uski.  

\item In genus $7$, Mathieu Pineault calculated numerically that the Fricke--Macbeath curve $F$ satisfies $\lambda_1(F) \approx 1.239$ \cite{Pineault} and it was shown that $\lambda_1(F) \in [1.23,1.26]$ in \cite{Lee}.

\item In genus $5$ to $6$, $14$, and $17$, we apply linear programming to some of the surfaces listed in \tableref{table:sys} to obtain lower bounds on their first eigenvalue based on their systole (see Section \ref{subsec:gaps}). The true value of $\lambda_1$ for these surfaces is certainly larger than the estimates we give since we discard all the geometric terms and the contribution of higher eigenvalues in the Selberg trace formula, while the test functions we use have only finitely many zeros. 

\item If $X \to Y$ is a finite-sheeted covering of hyperbolic orbifolds, then $\lambda_1(Y) \geq \lambda_1(X)$ since any eigenfunction on $Y$ lifts to an eigenfunction on $X$ with the same eigenvalue. We will use this in the next bullet point. 

\item For the entries equal to $1/4$ in the table, we use the fact that Selberg's conjecture is known to hold for the congruence subgroups $\Gamma_1(N)$ of square-free level $N < 857$ \cite{BookerStrombergsson}. Since $\Gamma_1(N) < \Gamma_0(N)$ as a finite-index subgroup, the conjecture also holds for $\Gamma_0(N)$ for the same levels. If $\Gamma_j(N)$ is torsion-free, then $X_j(N):=\HH^2 / \Gamma_j(N)$ ($j=0,1$) has no cone points and we can join its cusps in pairs to create thin handles following \cite{BBD}. By the results in that paper, the spectrum of the plumbed surface will be close to that of $X_j(N)$ (which has a discrete spectrum and a continuous spectrum equal to $[1/4,\infty)$ like all cusped surfaces). Since the spectral gap of $X_j(N)$ is $1/4$ for square-free $N<857$ \cite{BookerStrombergsson}, the plumbed surfaces thus have $\lambda_1$ as close to $1/4$ as we wish in these cases. The genus of the plumbed surfaces is equal to the genus of $X_j(N)$ plus half its number of cusps. Table \ref{table:compactifications} shows which congruence group we use for each genus concerned. The fact that these groups are indeed torsion free and that their signatures are as listed can be found in \cite[Section 4.2]{Miyake}. This information about congruence groups is also implemented in \texttt{Sage}.

\item There are other well-known ways of proving lower bounds on $\lambda_1$. The first of these is Cheeger's inequality $\lambda_1 \geq h^2/4$ where $h$ is the Cheeger constant \cite{Cheeger}. In high genus, this cannot be used to prove a lower bound of more than $\frac{1}{\pi^2}\approx 0.1013\ldots $ on $\lambda_1$ \cite{cheeger_constant}. However, it is not clear what bounds this approach might give in low genus as there are no explicit calculations of Cheeger constants for closed surfaces yet \cite{AdamsMorgan,Benson,BensonLakelandThen}. The second approach is to use the Jacquet--Langlands correspondence \cite{JacquetLanglands}, which allows one to derive lower bounds on the first eigenvalue of certain compact arithmetic surfaces from lower bounds on the first discrete eigenvalue of corresponding congruence covers of the modular curve (see \cite[Example 8.27]{Bergeron} and \cite{Hejhal} for concrete examples). We are not aware of any examples where both of the following hold: a better lower bound than $1/4$ is known for the cusped surface (see \cite{Huxley,BSV} for examples) and the Jacquet--Langlands correspondence gives rise to a closed surface of genus at most $20$ without cone points.
\end{itemize}

{\footnotesize
\begin{table}[ht]
\begin{center}
\caption{Some congruence subgroups of the modular group, their signatures, and the genus of their plumbing.}\label{table:compactifications}
\begin{tabular}{|c|c|c|}
\hline
$\Gamma$  & $(g,n)$  & $g_{\mathrm{plumbed}}$ \\
 \hline 
  $\Gamma_1(13)$ & $(2,12)$ & 8\\
  $\Gamma_1(15)$ & $(1,16)$ & 9\\
  $\Gamma_0(107)$ & $(9,2)$ & 10\\
  $\Gamma_0(87)$ & $(9,4)$ & 11\\
  $\Gamma_0(86)$ & $(10,4)$ & 12\\
  $\Gamma_1(17)$ & $(5,16)$ & 13\\
  $\Gamma_0(78)$ & $(11,8)$ & 15\\
  $\Gamma_1(19)$ & $(7,18)$ & 16\\
  $\Gamma_0(134)$ & $(16,4)$ & 18\\
  $\Gamma_0(102)$ & $(15,8)$ & 19\\
$\Gamma_0(227)$ & $(19,2)$ & 20\\
\hline
\end{tabular}
\end{center}
\end{table}
}

\subsection{Asymptotics}
A theorem of Cheng \cite[Theorem 2.1]{Cheng} states that
\begin{equation} \label{eq:Cheng}
\lambda_1(M) \leq \lambda_0(D_{\diam(M)/2})
\end{equation}
for any closed hyperbolic surface $M$, where $D_R$ is a hyperbolic disk of radius $R$, $\lambda_0(\Omega)$ is the smallest Dirichlet eigenvalue of $\Omega$, and $\diam(M)$ is the diameter of $M$. From this, Cheng deduces \cite[Corollary 2.3]{Cheng} the more explicit bound
\[
\lambda_1(M) \leq \frac{1}{4}+\left( \frac{4\pi}{\diam(M)}\right)^2.
\]
However, this can be improved using an inequality of Gage \cite[Theorem 5.2(a)]{Gage} on the smallest eigenvalue of hyperbolic disks, which states that
\begin{equation} \label{eqn:Gage}
\lambda_0(D_R) \leq \frac14 + \frac{\pi^2}{R^2} - \frac{1}{4\sinh^2(R)}.
\end{equation}
If we ignore the last term (of smaller order), we obtain the improved inequality
\[
\lambda_1(M) \leq \frac{1}{4}+\left( \frac{2\pi}{\diam(M)}\right)^2.
\]
In turn, the best known lower bound on the diameter is Bavard's bound
\[
\diam(M) \geq \arccosh\left( \frac{1}{\sqrt{3} \tan(\pi / (12g-6))} \right)
\]
where $g$ is the genus of $M$ \cite{Bavard}. Since the $\tan(x)\sim x$ as $x\to 0$ and $\arccosh(x)$ is asymptotic to $\log(x)$ as $x\to \infty$, Bavard's inequality has the same asymptotic behaviour as the more elementary inequalities
\[
\diam(M) \geq 2\arcsinh(\sqrt{g-1}) \geq \log(g-1)
\]
coming from area considerations, which result in
\[
\lambda_1(M) \leq \frac{1}{4}+\left( \frac{2\pi}{\log(g-1)}\right)^2.
\]
We will improve upon this by another factor of $4$.

\begin{thm} \label{thm:lamb_asymp}
There exists some $g_0 \geq 2$ such that every closed hyperbolic surface $M$ of genus $g\geq g_0$ satisfies
\[
\lambda_1(M) < \frac14 + \left(\frac{\pi}{\log(g)+0.7436}\right)^2.
\]
\end{thm}

\begin{rem}
An inequality of Savo \cite[Theorem 5.6]{Savo} states that 
\begin{equation} \label{eqn:Savo}
\lambda_0(D_R) \geq \frac14+ \frac{\pi^2}{R^2}-\frac{4 \pi^2}{R^3}
\end{equation}
for every $R>0$. It follows that the leading terms in Gage's upper bound \eqref{eqn:Gage} cannot be improved. Moreover, there exist sequences of hyperbolic surfaces $M_g$ of genus $g$ whose diameter is asymptotic to $\log(g)$ \cite{diameter}.  It follows that no multiplicative improvement as in \thmref{thm:lamb_asymp} could be obtained from Cheng's inequality \eqref{eq:Cheng}.
\end{rem}

\begin{rem}
As stated in the introduction, it is still unknown if there exist surfaces with $\lambda_1(M) \geq 1/4$ in every genus. However, it was proved recently that for any $\eps>0$, there exist surfaces with $\lambda_1(M) > 1/4 - \eps$ in large enough genus \cite{HideMagee} (see also \cite{AnantharamanMonk2,polynomial_rate,MageePudervanHandel}).
\end{rem}

The proof of \thmref{thm:lamb_asymp} will require the following limit calculation.

\begin{lem} \label{lem:limit}
We have
\[
\lim_{R\to \infty} R^4 \int_0^\infty \frac{\sin^2(\pi R x)}{1 -(Rx)^2} \frac{x \tanh(\pi x)}{(R x)^2}\,dx = \frac{\pi^2}{2}\int_0^\infty \frac{\sinh(x)}{x\cosh^3(x)}dx > 4.20718.
\]
\end{lem}

\begin{proof}
The following proof is due to an anonymous referee. Our goal is to simplify
\begin{align*}
I(R) &:= R^4\int_0^\infty \frac{\sin^2(\pi Rx)}{1-(Rx)^2} \frac{x\tanh(\pi x)}{(Rx)^2} dx \\ 
&= R^2\int_0^\infty \frac{\sin^2(\pi y)}{1-y^2} \frac{\tanh(\pi y / R)}{y} dy \\
& =  \frac{R^2}{4}\int_{-\infty}^\infty \sin^2(\pi y)\tanh(\pi y / R)\left(\frac{2}{y}-\frac{1}{y+1}-\frac{1}{y-1}\right) dy
\end{align*}
where we set $y=Rx$ and used the fact that the integrand is even. After splitting the integral as a sum of three integrals and making the changes of variable $u = y \pm 1$ in two of them, we obtain
\[
I(R)  = \frac{R^2}{4} \int_{-\infty}^\infty \frac{\sin^2(\pi y)}{y}\;\Big(2\tanh(\pi y /R) - \tanh(\pi(y-1)/R) - \tanh(\pi (y + 1) /R)\Big) dy
\]
since $\sin^2(\pi (y\pm 1))=\sin^2(\pi y)$ for all $y\in \RR$. The addition formulas for $\sinh$ and $\cosh$ yield
\[
2\tanh(a)+\tanh(a-b)+\tanh(a+b) = \frac{2\tanh(a)\sinh^2(b)}{\cosh(a-b)\cosh(a+b)}
\]
for any $a,b\in \RR$, so that
\begin{align*}
I(R)& = \frac{R^2 \sinh^2(\pi/R)}{2} \int_{-\infty}^\infty \frac{\sin^2(\pi y)}{y} \frac{\tanh(\pi y/R)}{\cosh(\pi (y - 1)/R)\cosh(\pi (y +1)/R)} dy \\
&= \frac{(R \sinh(\pi/R))^2}{2} \int_{-\infty}^\infty \sin^2(Rx)\frac{\tanh( x)}{x\cosh( x - \pi/R)\cosh( x + \pi/R)} dx \\
&= \frac{(R \sinh(\pi/R))^2}{4} \int_{-\infty}^\infty (1-\cos(2Rx))\phi_R(x) \,dx,
\end{align*}
where we set $x=\pi y /R$ and  $\phi_R(x):=\tanh(x)/(x\cosh(x - \pi/R)\cosh(x + \pi/R))$. Integration by parts and the triangle inequality shows that
\[
\left|\int_{-\infty}^\infty \cos(2Rx)\phi_R(x) \, dx \right| = \left|\frac{1}{2 R}\int_{-\infty}^\infty \sin(2 Rx)\phi_R'(x) \, dx \right| \leq \frac{\| \phi_R'(x) \|_1}{2 R}
\]
and this tends to zero as $R\to \infty$. Indeed, it is not difficult to bound $|\phi_R'|$ by an integrable function that does not depend on $R$.

We conclude that
\begin{align*}
\lim_{R\to\infty} I(R) &= \lim_{R\to\infty}\frac{(R \sinh(\pi/R))^2}{4}\lim_{R\to\infty}\int_{-\infty}^\infty \phi_R(x) \,dx \\
& = \frac{\pi^2}{4} \int_{-\infty}^\infty \frac{\tanh(x)}{x \cosh^2(x)} \,dx =  \frac{\pi^2}{2} \int_0^\infty \frac{\sinh(x)}{x \cosh^3(x)} \,dx,
\end{align*}
where we used the dominated convergence theorem ($\phi_R$ is dominated by $\frac{\tanh x}{x\cosh x}$) and the parity of the integrand. 

To get a rigorous lower bound this last integral, we observe that the integrand is positive so that
\[
\frac{\pi^2}{2}\int_0^\infty \frac{\sinh(x)}{x\cosh^3(x)}dx  \geq \frac{\pi^2}{2}\int_A^B \frac{\sinh(x)}{x\cosh^3(x)}dx 
\]
for any $0<A<B$. For $A = 1/10^6$ and $B=500$, interval arithmetic in \texttt{SageMath} certifies that the right-hand side is at least  $4.20718596495552>4.20718$.
\end{proof}

We can now prove the theorem.

\begin{proof}[Proof of \thmref{thm:lamb_asymp}]
The idea of the proof is to apply \thmref{thm:lambda} with functions of the form $f_R(x)=f(x/R)/R$ for some fixed non-negative admissible function $f$ such that $\what{f}$ is non-positive on $[1,\infty)$, so that $\what{f_R}(x) = \what{f}(Rx)$ is non-positive on $[1/R,\infty)$. Then $R=R(g)$ must be chosen as large as possible such that the inequality
\begin{equation} \label{eq:integral_ineq}
\what{f}(Ri/2) \leq 2(g-1)\int_0^\infty \what{f}(Rx) x \tanh(\pi x)\,dx
\end{equation}
remains valid. 

We choose
\[
f(x) = \sqrt{\frac{\pi}{8}}(2\pi - |x| +\sin|x|)\chi_{[-2\pi,2\pi]}(x) \]
whose Fourier transform is equal to
 \[
 \what{f}(x) = \frac{\sin^2(\pi x)}{x^2(1-x^2)}.
\]
Note that $f$ is non-negative on $\RR$ and $\what{f}$ is non-positive on $[1,\infty)$, as required. We also have that $f$ is admissible since $\what{f}$ is entire and is $O(|x|^{-4})$ on any horizontal strip of finite height.

For a given genus $g \geq 2$, we want to find an $R=R(g)>0$ such that inequality \eqref{eq:integral_ineq} holds. We first compute
\[
\what{f}(Ri/2)=  \frac{4\sinh^2(\pi R /2)}{R^2(1+R^2/4)} < \frac{4 e^{\pi R}}{R^4}.
\]
For the integral term, we have
\[
\lim_{R \to \infty} R^4\int_0^\infty \what{f}(Rx)x \tanh(\pi x)\,dx = \frac{\pi^2}{2}\int_0^\infty \frac{\sinh(x)\cosh(x)-x}{x^3\cosh^2(x)}dx
\]
by \lemref{lem:limit}. If $c$ is any positive number strictly smaller than the right-hand side (such as $4.2071$), then we have
\begin{equation} \label{eq:combined_estimate}
\frac{\what{f}(Ri/2)}{\int_0^\infty \what{f}(Rx)x \tanh(\pi x)\,dx} < \frac{4}{c} e^{\pi R}
\end{equation}
provided that $R$ is large enough. If $g$ is sufficiently large, then we can take $R = \frac{\log(g-1)+\log(c/2)}{\pi}$ so that the right-hand side of \eqref{eq:combined_estimate} becomes equal to $2(g-1)$. Then $f_R$ satisfies the hypotheses of \thmref{thm:lambda}, which proves the upper bound $\lambda_1(M) \leq L$ for every closed hyperbolic surface $M$ of large enough genus $g$, where $L$ is such that \[\sqrt{L - 1/4} = 1/R,\] the point after which $\what{f_R}$ stays non-positive. This gives 
\[
L = \frac{1}{4} + \frac{1}{R^2} = \frac{1}{4} + \frac{\pi^2}{(\log(g-1)+\log(c/2))^2}.
\]
Since $\log(4.20718/2)> 0.7436$ and $\log(g) - \log(g-1)$ tends to zero as $g \to \infty$, the inequality
\[
\lambda_1(M) < \frac{1}{4} + \frac{\pi^2}{(\log(g)+0.7436)^2}
\]
holds for all closed hyperbolic surfaces of sufficiently large genus.
\end{proof}

\begin{rem}
The choice of $f$ in the above proof is not at all random. It is proportional to the function $\psi_{1/2}(j_{1/2} x)$ from \secref{sec:Bessel}. It can be shown that this function is optimal for the strategy we use. Indeed, the problem amounts to minimizing the growth of $\what{f}(Ri/2)$ among functions such that $\int_0^\infty \what{f}(Rx)x \tanh(\pi x) \,dx$ is eventually positive (so that equation \eqref{eq:integral_ineq} has any chance of being satisfied). By a change of variable and an application of the dominated convergence theorem (see the proof of \propref{prop:sublinear} below), we get that a necessary condition for this eventual positivity is that the second moment $\int_0^\infty \what{f}(x) x^2 \,dx$ is non-negative. Moreover, by the Paley--Wiener theorem, the growth of $\what{f}(Ri/2)$ is controlled by the support of $f$. The question thus boils down to minimizing the support of $f$ among non-negative even functions whose Fourier transform is non-positive outside $[-1,1]$ and has a non-negative second moment. According to \cite[Remark 1.2]{Gorbachev} (with $d=s=1$, $m=0$, and $f$ and $\what f$ interchanged), the optimal function for this problem is the one we used. 
\end{rem}

\section{Multiplicity of the first eigenvalue}

\subsection{The criterion}
We denote the multiplicity of the first positive eigenvalue of the Laplacian on a hyperbolic surface $M$ by $m_1(M)$. The following criterion for bounding $m_1$ was first stated and proved in \cite[Lemma 3.2]{Klein} in  a slightly more general form.

\begin{thm} \label{thm:mult}
Let $M$ be a closed hyperbolic surface of genus $g\geq 2$ and suppose that $f$ is an admissible function such that
\begin{itemize}
\item $f(x) \geq 0$ for all $x \in \RR$;
\item $\what{f}\left(\sqrt{\lambda-\frac14}\right) \leq 0 $ whenever $\lambda \geq \lambda_1(M)$;
\item $\what{f}\left(\sqrt{\lambda_1(M)-\frac14}\right) <0$.
\end{itemize}
Then
\[
m_1(M) \leq  \frac{\what{f}(i/2) - 2(g-1)\displaystyle\int_0^\infty \what{f}(x) x \tanh(\pi x)\, dx}{-\what{f}\left(\sqrt{\lambda_1(M)-\frac14}\right)}.
\]
\end{thm}
\begin{proof}
Let us write $r_j(M) = \sqrt{\lambda_j(M)-\frac14}$. The Selberg trace formula tells us that
\begin{align*}
\what{f}(i/2) + m_1(M) \what{f}(r_1(M))  & \geq \sum_{j=0}^\infty \what{f}(r_j(M))  \\
&= 2(g-1) \int_0^\infty \what{f}(x) x \tanh(\pi x) \,dx  + \frac{1}{\sqrt{2\pi}} \sum_{\gamma \in \calC(M)} \frac{\Lambda(\gamma)f(\ell(\gamma))}{2 \sinh(\ell(\gamma)/2)}  \\
& \geq 2(g-1) \int_0^\infty \what{f}(x) x\tanh(\pi x) \,dx
\end{align*}
since $\what{f}(r_j(M))$ is non-positive for $j\geq 2$ and $f$ is non-negative on the length spectrum of $M$. Rearranging yields the desired result.
\end{proof}

Observe that since we require $f$ to be non-negative on $\RR$ and not constant equal to zero, we automatically have that $\what{f}$ is positive on the imaginary axis. In particular,  \thmref{thm:mult} cannot be used to prove bounds on the multiplicity of $\lambda_1$ for surfaces with $\lambda_1(M) \leq 1/4$, because then $\sqrt{\lambda_1(M) - 1/4}$ is on the imaginary axis. When applying the linear programming method numerically, it also seems that the resulting bounds tend to infinity as $\lambda_1(M)$ decreases to $1/4$ in any fixed genus.

We thus have to use different methods in order to bound $m_1(M)$ when $\lambda_1(M)$ is close to the interval $[0,1/4]$. When $\lambda_1(M) \in [0,1/4]$, a theorem of Otal \cite{Otal} (later generalized in \cite{OtalRosas}) says that $m_1(M) \leq 2g-3$. If $\lambda_1(M)$ is a little bit beyond $1/4$, then the bound gets slightly worse, namely, we have
\[
m_1(M) \leq  \begin{cases}2g-1 & \text{if }\lambda_1(M) \in (1/4,a_g] \\ 2g & \text{if } \lambda_1(M) \in (a_g,b_g)\end{cases}
\]
where $a_g$ (resp. $b_g$) is the smallest eigenvalue of the Laplacian on a hyperbolic disk of area $4\pi(g-1)$ (resp. $2\pi(g-1)$) subject to Dirichlet boundary conditions \cite[Theorem 1.1]{Klein}. Estimates for $a_g$ and $b_g$ are given in \cite[Section 2]{Klein}.

We thus use a combination of linear programming bounds and the above inequalities to bound $m_1$ in a given genus since we need to consider all possible values for $\lambda_1$. Similarly as for kissing numbers, when applying linear programming bounds over an interval $I$ of values for $\lambda_1$, we need to subdivide it into smaller intervals and use a single function $f$ on each subinterval $J$, taking care to bound $\what{f}\left(\sqrt{\lambda-\frac14}\right)$ for $\lambda \in J$.

\subsection{Low genus}  \label{subsec:multiplicity_examples}

The bounds we have obtained on $m_1$ for $g$ between $2$ and $20$ are listed in \tableref{table:mult}.  They improve upon the previous best upper bound of $2g+3$ from \cite{Sevennec} (which applies to all Schr\"odinger operators on Riemannian surfaces). In genus $2$ and $3$, our bounds were previously obtained in \cite{Klein} and we do not repeat these calculations in the ancillary file \texttt{verify\_multiplicity.ipynb} that certifies the other values. 

We now discuss lower bounds. In  every genus $g\geq 3$, Colbois and Colin de Verdière \cite{CCV} constructed closed hyperbolic surfaces satisfying
$
m_1(M) = \left\lfloor \frac12 \left( 1 + \sqrt{8g+1}\right) \right\rfloor.
$
This has the same order of growth as the maximum of $\left\lfloor \frac12\left( 5 + \sqrt{48g+1} \right) \right\rfloor$ among all closed, connected, orientable Riemannian surfaces of genus $g$  conjectured by Colin de Verdière \cite{CdV86,CdV}. We list these conjectured values in the table for comparison. As mentioned in the introduction, counterexamples to Colin de Verdière's upper bound have since been found in genus $10$ and $17$, with multiplicity $16$ and $21$ respectively \cite{counterexamples}.

In genus $2$, Colin de Verdière's formula comes out to $7$ but our upper bound is $6$. The conjectured maximum among hyperbolic surfaces is $3$.

{\footnotesize
\begin{table}[htp] 
\centering
\caption{Bounds on the maximum of the multiplicity $m_1$.} \label{table:mult}
\begin{tabular}{|l|l|l|l|l|}
\hline
genus & lower bound & conjecture & LP bound & previous bound \\
\hline
2 & 3 \cite{Klein} & 3 {\footnotesize(hyperbolic)}, 7 {\footnotesize (Riemannian)} & 6 \cite{Klein} & 7 \cite{Sevennec}  \\ 
3 & 8 \cite{Klein} & 8 & 8 \cite{Klein}  & 9 \cite{Sevennec} \\ 
4 & 4 (\S\ref{subsec:multiplicity_examples}) or 6? \cite{Cook} &9& 10  & 11 \cite{Sevennec} \\ 
5 & 3 \cite{CCV} &10& 11 & 13 \cite{Sevennec} \\ 
6 &  4  \cite{CCV} &11& 13 & 15 \cite{Sevennec}\\ 
7 &  7 (\S\ref{subsec:multiplicity_examples} and \cite{Lee})  &11& 15 & 17 \cite{Sevennec}\\
8 &  6 (\S\ref{subsec:multiplicity_examples}) &12& 17 & 19 \cite{Sevennec}\\ 
9 &  4 \cite{CCV}&12& 19 & 21 \cite{Sevennec}\\ 
10 & 8 (\S\ref{subsec:multiplicity_examples}) &13& 20 & 23 \cite{Sevennec}\\ 
11 & 5 \cite{CCV}&14& 22 & 25 \cite{Sevennec}\\ 
12 & 5 \cite{CCV}&14& 24 & 27 \cite{Sevennec}\\ 
13 & 5 \cite{CCV}&15& 26 & 29 \cite{Sevennec}\\ 
14 & 12 (\S\ref{subsec:multiplicity_examples}) &15& 28 & 31 \cite{Sevennec}\\ 
15 & 7 (\S\ref{subsec:multiplicity_examples}) &15& 30 & 33 \cite{Sevennec}\\ 
16 & 8 (\S\ref{subsec:multiplicity_examples}) &16& 32 & 35 \cite{Sevennec}\\ 
17 & 6 \cite{CCV} &16& 34 & 37 \cite{Sevennec}\\
18 & 6 \cite{CCV} &17& 36 & 39 \cite{Sevennec}\\  
19 & 8 (\S\ref{subsec:multiplicity_examples}) &17& 38 & 41 \cite{Sevennec}\\ 
20 & 6 \cite{CCV} &18& 40 & 43 \cite{Sevennec}\\ 
\hline
\end{tabular}

\end{table}
}

Colbois and Colin de Verdière modelled their examples on graphs and used a transversality argument to control the multipli\-city. Another way to obtain lower bounds on multiplicity is to use representation theory \cite{Jenni,BurgerColbois,Cook,Klein} since the isometry group of a closed hyperbolic surface is finite and acts on the eigenspaces of the Laplacian. This means that if all the irreducible representations of a group have dimension at least $d$, then the multiplicity of any eigenvalue is at least $d$. The problem is that there is always the trivial representation of dimension $1$, so one must find a way to rule out $1$-dimensional real representations from appearing in eigenspaces. Proposition 4.4 in \cite{Klein} gives such a criterion for kaleidoscopic surfaces, as defined below.

Given integers $2\leq p \leq q \leq r$, a \emph{$(p,q,r)$-triangle surface} (sometimes called quasiplatonic) is a hyperbolic surface of the form $\HH^2/ \Gamma$ for some finite-index normal subgroup $\Gamma$ of the $(p,q,r)$-triangle group, that is, the group generated by rotations of order $p$, $q$, and $r$ around the vertices of a hyperbolic triangle with interior angles $\frac{\pi}{p}$, $\frac{\pi}{q}$, and $\frac{\pi}{r}$ at the corresponding vertices. A \emph{kaleidoscopic surface} is defined similarly but with the extended triangle group generated by the reflections in the sides of a $(p,q,r)$-triangle.

A hyperbolic surface that admits an orientation-reversing isometry is called \emph{symmetric}, reflexible, or real. Perhaps surprisingly, not every triangle surface is symmetric. In fact, the two Hurwitz surfaces in genus $17$ are not \cite[Theorem 5]{Singerman} (these are isometric but via an orientation-reversing isometry, so they are considered as distinct). It is therefore desirable to have a criterion for ruling out $1$-dimensional real representations for asymmetric surfaces. Even for kaleidoscopic surfaces, the representation theory of $\Isom^+$ can behave better than that of $\Isom$ in some cases. We thus prove the following variant of \cite[Proposition 4.4]{Klein}.

\begin{lem} \label{lem:no1rep}
If $M$ is a hyperbolic $(p,q,r)$-triangle surface of area at least \[6\pi r \left(1 -\frac1p - \frac1q - \frac1r \right),\] then no $1$-dimensional real representations of $\Isom^+(M)$ can occur in the eigenspace corresponding to $\lambda_1(M)$.
\end{lem}
\begin{proof}
Suppose that $f$ is an eigenfunction contained in a $1$-dimensional representation of $\Isom^+(M)$ contained in the $\lambda_1$-eigenspace. Then $\Isom^+(M)$ acts by multiplication by $\pm 1$ on $f$ so the set $f^{-1}(0)$ is invariant. By Courant's nodal domain theorem, the complement of this set (which is a union of analytic curves intersecting transversely \cite{ChengMultiplicity}) has exactly two connected components, so in particular $f^{-1}(0)$ is non-empty. We will show that this leads to a contradiction.

Note that $\Isom^+(M) \geq G : = \Gamma / \pi_1(M)$ where $\Gamma$ is the $(p,q,r)$-triangle group but a priori the inclusion can be strict (because of inclusions between some triangle groups). We will work $G$ instead of $\Isom^+(M)$ because that makes things simpler. Let $T$ be any $(p,q,r)$-triangle used to define $\Gamma$, let $\widetilde\calT$ be the tiling of $\HH^2$ generated by the reflections in the sides of $T$, and let $\calT$ be the projection of $\widetilde \calT$ to $M$. Since $\Gamma$ acts simply transitively on adjacent pairs of triangles in $\widetilde \calT$ that share a particular kind of side (say joining the vertices of type $p$ and $q$), so does $G$ on $M$. In other words, $|G| = |\calT| / 2$ and
\[\area(M) = 2|G| \area(T)= 2|G| \pi\left(1 - \frac1p - \frac1q - \frac1r \right).\] From the hypothesis on $\area(M)$, we get $|G|\geq 3r$.

Let $Q = M/G$, let $\pi: M \to Q$ be the quotient map, and let $F = \pi(f^{-1}(0))$. We have that $Q$ is a sphere with three cone points and $F$ is a finite analytic graph with no isolated points and any vertices of degree $1$ are contained in the cone points of $Q$.

Suppose that a component $U$ of $Q \setminus F$ contains exactly one cone point $v$ of $Q$. Then $\pi^{-1}(U)$ has $|G|/k$ components, where $k$ is one of $p$, $q$, or $r$, all of which are nodal domains. By the above, we have $|G|/k \geq |G|/r \geq 3$, which contradicts Courant's theorem.

If $Q \setminus F$ has more than two connected components, then $f$ has more than two nodal domains (the preimages of these components). It follows that $F$ contains at most one cycle. 

Suppose that $F$ does contain a cycle. Then $Q\setminus F$ has two connected components by the Jordan curve theorem. Since each component contains either $0$, $2$, or $3$ cone points by the above argument, at least one of them, call it $U$, does not contain any. The quotient map $\pi : M \to Q$ is thus unbranched over $U$, so that $\pi^{-1}(U)$ has $|G|$ components, all of which are nodal domains. This is again a contradiction since $|G| > 2$.

We conclude that $F$ is a forest and in particular its complement is connected. Recall that all the leaves of $F$ are contained in the $3$ cone points. Moreover, all the cone points must belong to $F$ since it has at least two leaves and if it has only two, then its complement contains exactly one cone point, which is impossible by the above reasoning. We conclude that $F$ is a tree. Once again, the map $\pi$ is a covering map over the simply connected domain $Q\setminus F$, so $\pi^{-1}(Q \setminus F)$ has $|G|>2$ components, contradiction.
\end{proof}

The list of all triangle surfaces in genus $2$ to $101$ was tabulated by Conder using \texttt{Magma} and is available at \cite{ConderList}. We went through all examples in genus $2$ to $20$, verified if the area hypothesis of \cite[Proposition 4.4]{Klein} or of \lemref{lem:no1rep} was satisfied, calculated the character tables for $\Isom^+$ and $\Isom$ using \texttt{Sage/GAP}, and then calculated the second smallest dimension of an irreducible real representation of the group (or rather, a lower bound for it). The cases where the resulting multiplicity is larger than the $\left\lfloor \frac12 \left( 1 + \sqrt{8g+1}\right) \right\rfloor$ obtained in \cite{CCV} are as follows:
  
\begin{itemize}
\item The Bolza surface of genus $2$ and type $(2,3,8)$ with $m_1=3$. Jenni's proof of this fact in \cite{Jenni} contained an error partially corrected in \cite{Cook} and fixed in \cite{Klein}. In this case, $\Isom$ admits some $2$-dimensional irreducible real representations but they can be ruled out with a more careful analysis. 
\item The Klein quartic of genus $3$ and type $(2,3,7)$ with $m_1=8$ \cite{Klein} (the representation theory of $\Isom \cong \PGL(2,7)$ only gives $m_1\geq 6$ \cite{Cook}).
\item Bring's curve $B$ of genus $4$ and type $(2,4,5)$, which satisfies $\Isom^+(B) \cong S_5$ and $\Isom(B) \cong S_5 \times (\ZZ/2\ZZ)$. While $\Isom(B)$ admits some $2$-dimensional irreducible real representations, $\Isom^+(B) \cong S_5$ has two real $1$-dimensional representations and then irreducible complex representations of dimensions $4$, $5$, and $6$. In particular, its irreducible real representations of dimension more than $1$ have real dimension at least $4$. By \lemref{lem:no1rep}, $1$-dimensional real representations of $\Isom^+(B)$ cannot appear in the first eigenspace so we have $m_1(B) \geq 4$. Numerical evidence suggests that the correct value is $m_1(B)=6$ \cite{Cook}.

\item The Fricke--Macbeath curve $F$ of genus $7$ is a $(2,3,7)$-triangle surface such that $\Isom^+(F) \cong \PSL(2,8) = \SL(2,8)$ (see e.g. \cite{Singerman}). The non-trivial irreducible complex representations of this group have complex dimensions $7$, $8$, and $9$ \cite{CharTables}. It follows from \lemref{lem:no1rep} that $m_1(F) \geq 7$. It was shown in \cite{Lee} (see also \cite{Pineault}) that equality holds.

\item The $(2,3,8)$-triangle surface of genus $8$ labelled T8.1 in Conder's list satisfies $m_1 \geq 6$ due to the representation theory of $\Isom^+$. 

\item The $(2,4,5)$-triangle surface of genus $10$ labelled T10.7 in Conder's list satisfies $m_1 \geq 8$ due to the representation theory of $\Isom$. 

\item The $(2,3,7)$-triangle surface of genus $14$ labelled T14.1 in Conder's list satisfies $m_1 \geq 12$ due to the representation theory of $\Isom$. 

\item The $(2,3,9)$-triangle surface of genus $15$ labelled T15.1 in Conder's list satisfies $m_1 \geq 7$ due to the representation theory of $\Isom^+$. 

\item The $(2,3,8)$-triangle surface of genus $16$ labelled T16.1 in Conder's list satisfies $m_1 \geq 8$ due to the representation theory of $\Isom^+$. 

\item The $(2,4,5)$-triangle surface of genus $19$ labelled T19.3 in Conder's list satisfies $m_1 \geq 8$ due to the representation theory of $\Isom$. 
\end{itemize}

The ancillary file \texttt{verify\_multiplicity\_examples.ipynb} contains the computer code that found these examples. Interestingly, in most genera the bound of Colbois and Colin de Verdière is matched by some triangle surface. However, since the sum of the squares of the degrees of the irreducible representations of a group is equal to the order of the group, and since the isometry group of a closed hyperbolic surface of genus $g$ has order at most $168(g-1)$ \cite{Hurwitz}, the best this method could give is still $O(\sqrt{g})$.

Back to upper bounds, note that the linear programming bounds listed in \tableref{table:mult} never go below $2g$ in the range considered here, so we can assume that $\lambda_1$ is between $b_g$ and the upper bounds from \tableref{table:lambda} to prove them. For genus $6$ onwards, our bounds are worse when $\lambda_1$ is close to (the estimate for) $b_g$ than at the upper bound on $\lambda_1$. For example in genus $20$, we obtain an upper bound of $23$ instead of $40$ when $\lambda_1 \in [0.750384 , 0.91139 ]$. Our intuition is that $m_1$ should be maximized among local maximizers of $\lambda_1$. The fact that the bound on $m_1$ increases when $\lambda_1$ decreases is an artefact of the method and is not necessarily representative of the reality.

\subsection{Asymptotics}

In higher genus, we will decrease Sévennec's upper bound by $4$. We start by proving a sublinear bound on $m_1$ under the assumption that $\lambda_1$ is fairly large. In the statement below, $j_\alpha$ stands for the first positive zero of the Bessel function $J_\alpha$. The $k$-th positive zero is denoted $j_{\alpha,k}$. 

\begin{prop}  \label{prop:sublinear}
For every $p_1,p_2 \in (j_0,\pi]$, there exists a constant $C$ and some $g_0 \geq 2$ such that if $g\geq g_0$ and $M$ is a closed hyperbolic surface of genus $g$ with
\[
\lambda_1(M) \in \left[\frac14 +\left(\frac{p_1}{\log(g)}\right)^2 , \frac14  + \left(\frac{p_2}{\log(g)}\right)^2\right],
\]
then $m_1(M) \leq C g/\log(g)^3$.
\end{prop}
\begin{proof}
We start by picking two constants $\alpha$ and $c$ that only depend on $p_1$ and $p_2$ (and not on $g$).

For every $k \in \NN$, the $k$-th zero $j_{\alpha,k}$ of $J_\alpha$ is increasing as a function of $\alpha$ provided that $\alpha\geq 0$ \cite[p.507]{Watson}. Together with the interlacing property of the Bessel zeros \cite[p.479]{Watson}, this implies that
\[
j_{0,2}>j_{1,1}>j_{1/2,1}=\pi > j_{0,1} = j_0.
\]
In fact, much better numerics are known, namely, $j_0 \approx 2.4048$ and $j_{0,2}\approx 5.5201$ \cite{Davis}. In particular, we have
\[
\frac{p_1}{p_2} > \frac{j_0}{\pi} > \frac{j_0}{j_{0,2}}.
\]
Since $j_{\alpha,k}$ depends continuously on $\alpha$ for every $k$ \cite[p.507]{Watson}, there exists some $\alpha$ in $(0,1/2)$ such that
\[
 p_1 > j_\alpha \quad \text{and} \quad \frac{p_1}{p_2} > \frac{j_\alpha}{j_{\alpha,2}}.
\]
We then take any $c \in (j_\alpha/p_1,\min\{j_{\alpha,2}/p_2,1\})$. This interval is not empty since $j_\alpha/p_1 < 1$ and $j_\alpha/p_1 < j_{\alpha,2}/p_2$.

We set $R = c \log(g)$ and apply \thmref{thm:mult} with the function $f$ such that 
\[
\what{f}(x) = \psi_\alpha(Rx) = \frac{J_\alpha(Rx)^2}{ (Rx)^{2\alpha} (1-R^2 x^2/j_\alpha^2)}.
\]
The function $\what{f}$ is positive-definite by \cite[Remark 1.1]{Gorbachev} and
\[
\mu:=\min\st{-\what{f}(x)}{x \in \left[ \frac{p_1}{\log(g)},\frac{p_2}{\log(g)}\right]} = \min\st{- \psi_\alpha(y)}{y \in [c\,p_1,c\,p_2]}
\]
is a positive constant that only depends on our choice of parameters, but not on the genus. Indeed, $\psi_\alpha$ is continuous and strictly negative on $(j_\alpha , j_{\alpha,2})\supset [c\,p_1,c\,p_2]$.

By the asymptotic behaviour of Bessel functions along the imaginary axis \eqref{eq:Bessel_assymptotic_imag}, we have
\[
\what{f}(i/2) \sim \frac{1}{\pi} \frac{e^{R}}{R^{2\alpha+1}(1+R^2/(4 j_\alpha^2))} = o\left(\frac{g}{\log(g)^3}\right)
\]
as $g\to \infty$.

For the integral term, we make a change of variable to find that
\begin{align*}
\left|\int_0^\infty \what{f}(x) x \tanh(\pi x) \,dx \right| & = \left|\int_0^\infty \psi_\alpha(Rx) x \tanh(\pi x) \,dx \right| \\
& =  \left|\frac{\pi}{R^3}\int_0^\infty  \psi_\alpha(y) y^2 \frac{\tanh(\pi y/R)}{\pi y/R} \,dy \right| \\
& \leq  \frac{\pi}{R^3}\int_0^\infty  |\psi_\alpha(y)| y^2  \,dy \\
&= O\left(\frac{1}{\log(g)^3}\right).
\end{align*}
Indeed, $|\psi_\alpha(y)| y^2$ is integrable by the asymptotic estimate \eqref{eq:Bessel_assymptotic_real} since $\alpha>0$. 

By Theorem \ref{thm:mult} applied with the function $f$, we obtain
\[
m_1(M) \leq \frac{1}{\mu} \left( \what{f}(i/2) -2(g-1) \int_0^\infty \what{f}(x) x \tanh(\pi x) dx\right) = O\left(\frac{g}{\log(g)^3}\right)
\]
if $\sqrt{\lambda_1(M)-1/4} \in \left[ \frac{p_1}{\log(g)},\frac{p_2}{\log(g)}\right]$ and if $g$ is sufficiently large.
\end{proof}

\begin{rem}
It is easy to modify the proof of \thmref{thm:mult} to bound the total number of eigenvalues in an interval $[a,b]$ assuming that $\lambda_1(M)\in [a,b]$ (see \cite[Lemma 3.2]{Klein}). This more general version of the criterion implies that  under the same hypotheses as above, the total number of eigenvalues contained in the given interval is bounded by the same quantity $C g / \log(g)^3$.
\end{rem}

 See \cite[Corollary 1.7]{GLST} for a similar result showing that for any compact interval $I$ contained in $[0,1/4)\cup(1/4,\infty)$, the multiplicity of any eigenvalue in $I$ grows at most sublinearly with the genus $g$ with probability tending to $1$ as $g \to \infty$ with respect to the Weil--Petersson measure. See also \cite[Corollary 6]{Monk}. An upper bound on $m_1$ of the form $O(g/\log g)$ was recently proved in \cite{LetrouitMachado} for negatively pinched surfaces with systole bounded away from zero.

We then combine the sublinear upper bound from \propref{prop:sublinear} with a previous (linear) upper bound from \cite{Klein} for slightly smaller $\lambda_1$ to obtain a global upper bound on $m_1$ in large genus.

\begin{thm}\label{thm:mult_asymp}
There exists some $g_0 \geq 2$ such that every closed hyperbolic surface $M$ of genus $g\geq g_0$ satisfies
\[
m_1(M) \leq 2g - 1.
\]
\end{thm}

\begin{proof} By Theorem \ref{thm:lamb_asymp}, we can assume that 
\[
\lambda_1(M) \leq \frac14 + \left(\frac{\pi}{\log(g)+0.7436}\right)^2 \leq \frac14 + \left(\frac{\pi}{\log(g)}\right)^2.
\]
Now pick any $p\in (j_0,\pi)$. \propref{prop:sublinear} implies (a stronger version of) the result if 
\[
\lambda_1(M) \geq \frac14 + \left(\frac{p}{\log(g)}\right)^2.
\]

It remains to consider the case where $\lambda_1(M)$ is smaller than this bound. This case was already handled in \cite[Theorem 1.1]{Klein}, which states that $m_1(M) \leq 2g-1$ whenever $\lambda_1(M) \leq a_g$, where $a_g$ is the smallest Dirichlet eigenvalue on a hyperbolic disk of area $4\pi(g-1)$  and hence of radius $2\arcsinh(\sqrt{g-1})$. From Savo's inequality \eqref{eqn:Savo}, we have
\[
a_g - \frac{1}{4}  \geq \frac{\pi^2}{4\arcsinh(\sqrt{g-1})^2} - \frac{\pi^2}{2\arcsinh(\sqrt{g-1})^3} \sim \frac{\pi^2}{\log(g)^2} \quad \text{as } g\to \infty.
\]
In particular, when $g$ is large enough we have
\[
a_g > \frac{1}{4}+ \left(\frac{p}{\log(g)}\right)^2.
\]
As such, all possibilities for $\lambda_1(M)$ are covered and the inequality is proved.
\end{proof}

Note that the proof relied partly on \cite[Theorem 1.1]{Klein} whose proof is based on the results of Sévennec \cite{Sevennec}. Different methods would be required to prove sublinear upper bounds. 

\section{Small eigenvalues}

\subsection{The criterion}

We start by proving a general criterion for bounding the number $N_{[0,L]}(M)$ of eigenvalues of $\Delta_M$ in the interval $[0,L]$ for any $L>0$. 

\begin{thm} \label{thm:nsmall}
Let $L,R>0$ and suppose that $f: \RR \to \RR$ is an admissible function such that
\begin{itemize}
\item $f(x) \leq 0$ for every $x\geq R$;
\item $\what{f}(\xi) \geq 0$ for every $\xi\in \RR$;
\item $\what{f}(\sqrt{\lambda - 1/4}) \geq 1$ for every $\lambda \in [0,L]$;

\end{itemize}
Then every closed hyperbolic surface $M$ of genus $g\geq 2$ with $\sys(M)\geq R$ satisfies
\[
N_{[0,L]}(M) \leq 2(g-1)\int_0^\infty \what{f}(x) x \tanh(\pi x) \,dx +1 - \what{f}(i/2).
\]
\end{thm}
\begin{proof}
By hypothesis, $\what{f}(r_j(M))$ is at least $1$ for all eigenvalues $\lambda_j(M)$ in the interval $[0,L]$ and non-negative at all eigenvalues. Also note that $\lambda_0(M)=0\in [0,L]$ corresponds to the term $\what{f}(i/2)$. We thus have
\begin{align*}
N_{[0,L]}(M)-1+\what{f}(i/2) &\leq \sum_{j=0}^\infty \what{f}(r_j(M)) \\
&=  2(g-1)\int_0^\infty \what{f}(x) x \tanh(\pi x) \,dx + \frac{1}{\sqrt{2\pi}}\sum_{\gamma \in \calC(M)} \frac{\Lambda(\gamma)f(\ell(\gamma))}{2 \sinh(\ell(\gamma)/2)} 
\\ &\leq 2(g-1)\int_0^\infty \what{f}(x) x \tanh(\pi x) \,dx
\end{align*}
by the Selberg trace formula, where the last inequality is because the geometric terms are non-positive by hypothesis.
\end{proof}

Recall that an eigenvalue of the Laplacian on $M$ is \emph{small} if it belongs to the interval $[0,1/4]$. The number $\Nsmall(M)$ of small eigenvalues of $M$ is therefore equal to $N_{[0,1/4]}(M)$. Also note that $\lambda \in [0,1/4]$ if and only if $\sqrt{\lambda - 1/4} \in i [-1/2,1/2]$.
In practice, we will often use the weaker bound
\[
\frac{\Nsmall(M)}{2(g-1)} \leq \inf_f \int_0^\infty \what{f}(x) x \tanh(\pi x) \,dx
\] 
(where the infimum is over the functions $f$ that satisfy the hypotheses of the theorem) instead of the one given in \thmref{thm:nsmall} because the right-hand side only depends on $\sys(M)$ and not on the genus. A theorem of Otal and Rosas \cite{OtalRosas} states that the left-hand side is always bounded above by $1$, and this is sharp for surfaces with a very short pants decomposition \cite{Buser14} (see also \cite[Theorem 8.1.3]{Buser}). Our goal is therefore to find the smallest value of $R=\sys(M)$ for which the right-hand side becomes smaller than $1$ and to estimate how it decreases as the systole increases.

\subsection{Asymptotics}

We start with the asymptotics instead of the numerics for $\Nsmall$
because they are easier to obtain than for the other invariants and they give us a point of comparison for the numerics.

\begin{thm}\label{thm:small_asymp}
If $M$ is a closed hyperbolic surface of genus $g\geq 2$, then
\[
\Nsmall(M) <  \min\left( \frac{24 \pi^2(g-1)}{\sys(M)^3}, \frac{16(g-1)}{\sys(M)^2} \right).
\]
\end{thm}
\begin{proof}
Let $f: \RR \to \RR$ be an even, non-negative, admissible (hence continuous), positive-definite function supported in $[-1,1]$ normalized so that $\what{f}(0)=1$.

We then use the function $f_R(x)=f(x/R)/R$ in \thmref{thm:nsmall} with $L=1/4$ and $R =\sys(M)$. The first and second bullet points in the statement of the theorem are satisfied by hypothesis on $f$ while the third one follows from the non-negativity of $f$ (and hence of $f_R$). Indeed, 
\[
\what{f_R}(it) = \sqrt{\frac{2}{\pi}} \int_0^\infty f_R(x) \cosh(tx) \,dx \geq \sqrt{\frac{2}{\pi}} \int_0^\infty f_R(x) \,dx = \what{f_R}(0) = \what{f}(0) = 1
\]
for every $t \in \RR$, with equality only if $t=0$.

If $\what{f}$ has a finite first moment on $[0,\infty)$, then we can estimate
\begin{align*}
\int_0^\infty \what{f_R}(x) x \tanh(\pi x) \,dx &= \int_0^\infty \what{f}(Rx) x \tanh(\pi x) \,dx \\
&< \int_0^\infty \what{f}(Rx) x\,dx = \frac{1}{R^2}\int_0^\infty \what{f}(y) y\,dy.
\end{align*}
The function $f$ determined by $\what{f}(x) = \eta_\alpha(x/2)^2$, where
\[
\eta_\alpha(x) = 2^\alpha \Gamma(\alpha+1) \frac{J_\alpha(x) }{x^\alpha} 
\]
is the normalized Bessel function, satisfies all the necessary requirements provided that $\alpha>1/2$ (see Section \ref{sec:Bessel}). We then compute
\[
 \int_0^\infty \what{f}(y) y \, dy = \int_0^\infty \eta_\alpha(y/2)^2 y\,dy = 4 \int_0^\infty \eta_\alpha(x)^2 x\,dx.
\]
The recurrence formulae for Bessel functions \cite[p.2]{Watson} imply that
\[
\eta_\alpha'(x) = \frac{-x}{2(\alpha+1)} \eta_{\alpha+1}(x) = \frac{2\alpha}{x} (\eta_{\alpha-1}(x) - \eta_\alpha(x))
\]
and from this it is easy to check that
\[
\frac{1}{2-4\alpha}\left( 4\alpha^2 \eta_{\alpha-1}(x)^2 + \eta_\alpha(x)^2 x^2 \right)
\]
is a primitive of $\eta_\alpha(x)^2 x$ which vanishes at infinity. We thus have
\[
\int_0^\infty \what{f}(y) y \, dy = 4\int_0^\infty \eta_\alpha(x)^2 x\,dx = \frac{16 \alpha^2}{4 \alpha -2}
\]
since $\eta_\beta(0)=1$ for every $\beta$. This quantity is at least $8$ with equality only if $\alpha=1$ . For that parameter, the resulting inequality is
\[
\Nsmall(M) < \frac{16(g-1)}{\sys(M)^2}
\]
if we ignore the term $1-\what{f_R}(i/2) < 0$.

We can also write
\begin{align*}
\int_0^\infty \what{f_R}(x) x \tanh(\pi x )\, dx &= \int_0^\infty \what{f}(Rx) x \tanh(\pi x )\, dx \\
&< \pi \int_0^\infty \what{f}(Rx) x^2\, dx \\
& = \frac{\pi}{R^3} \int_0^\infty \what{f}(y) y^2 \, dy.
\end{align*}
provided that $\what{f}$ has a finite second moment.

With the same function $\what{f}(x) = \eta_\alpha(x/2)^2$ as before (but with $\alpha>1$ this time),  we need to compute
\[
\int_0^\infty \what{f}(y) y^2 \, dy =  \int_0^\infty \eta_\alpha(y/2)^2 y^2 \, dy = 8 \int_0^\infty \eta_\alpha(x)^2 x^2 \, dx.
\]
Integration by parts with $u = x$ and $dv = \eta_\alpha(x)^2 x \,dx$ yields
\[
\int_0^\infty \eta_\alpha(x)^2 x^2 \, dx = \frac{1}{4\alpha-2} \int_0^\infty \left(4 \alpha^2 \eta_{\alpha-1}(x)^2 + \eta_\alpha(x)^2x^2\right) dx
\]
and hence
\[
\int_0^\infty \eta_\alpha(x)^2 x^2 \, dx = \frac{4\alpha^2}{4\alpha-3} \int_0^\infty \eta_{\alpha-1}(x)^2 \,dx.
\]

Recall that
\[
\what{\eta_\beta}(t)=\frac{\sqrt{2\pi}}{B(\frac12,\beta+\frac12)} \rect(t/2) (1-t^2)^{\beta - 1/2}
\]
for every $\beta > 1/2$ and every $t \in \RR$, where $\rect$ is the characteristic function of $[-1/2,1/2]$ and $B$ is the Beta function. By the convolution formula, we have
\begin{align*}
\int_0^\infty \eta_\beta(x)^2 dx =  \sqrt{\frac{\pi}{2}} \, \what{\eta_\beta^2}(0) &= \frac12 \, \what{\eta_\beta}\ast \what{\eta_\beta}(0) \\
&= \frac12 \left(\frac{\sqrt{2\pi}}{B(\frac12,\beta+\frac12)}\right)^2 \int_{-1}^1 (1-t^2)^{2\beta-1} \,dt \\
&= \frac{2\pi}{B(\frac12,\beta+\frac12)^2} \int_0^1 (1-t^2)^{2\beta-1} \,dt \\
& =\frac{\pi}{B(\frac12,\beta+\frac12)^2} \int_0^1 u^{1/2 - 1}(1-u)^{2\beta-1} \,du \\
& = \frac{\pi B(\textstyle\frac12,2\beta)}{B(\frac12,\beta+\frac12)^2}
\end{align*}

Returning to the original problem, we have
\begin{align*}
\int_0^\infty \what{f}(y) y^2 \, dy  = 8 \int_0^\infty \eta_\alpha(x)^2 x^2 \, dx &= \frac{32\alpha^2}{4\alpha-3} \int_0^\infty \eta_{\alpha-1}(x)^2 dx \\
&= \frac{32\alpha^2 \pi B(\textstyle\frac12,2\alpha-2)}{(4\alpha-3)B(\frac12,\alpha-\frac12)^2}.
\end{align*}
This function is minimized at $\alpha = 3/2$, where it takes the value $12\pi$. The resulting inequality is
\[
\Nsmall(M) < \frac{24\pi^2(g-1)}{\sys(M)^3}.  \qedhere
\]
\end{proof}

\begin{rem}
Note that $24\pi^2 / x^3 < 16 / x^2$ if and only if $x>3\pi^2 / 2 \approx 14.8044$ (for $x>0$).
\end{rem}

We now compare our inequality
\[
\Nsmall(M) < \frac{24\pi^2(g-1)}{\sys(M)^3}
\]
with Huber's inequality
\[
\Nsmall(M) \leq \frac{3\pi^2 (g-1)}{8(\log (\cosh(\sys(M)/4)))^3}
\]
from \cite{HuberSmallEigs}. Since $\cosh(x) < e^x$ for every $x> 0$, we have $ \log( \cosh x) < x$ and hence
\[
\frac{3\pi^2 (g-1)}{8(\log (\cosh(\sys(M)/4)))^3} > \frac{3\pi^2 (g-1)}{8\sys(M)^3 / 64} = \frac{24\pi^2(g-1)}{\sys(M)^3},
\]
so that our bound is better but only slightly. Indeed, the inequality $\cosh(x) > e^x / 2$ implies that the factors that multiply $(g-1)$ in the two inequalities are asymptotic to each other as $\sys(M) \to \infty$ despite the fact that the two proofs use different functions (Huber uses Legendre functions while we use Bessel functions).

From our inequality
\[
\Nsmall(M) < \frac{16(g-1)}{\sys(M)^2}
\]
it follows that if $\sys(M)\geq \sqrt{8}$, then $\Nsmall(M)<2g-2$. In \cite{HuberSmallEigs}, Huber also proves the inequality
\[
\Nsmall(M) \leq \frac{g-1}{2 \log (\cosh(\sys(M)/4))},
\]  
which implies that $\Nsmall(M) < 2g-2$ as soon as 
\[\sys(M) > 4\arccosh(e^{1/4})\approx 2.947618.\]
This is better than the constant $3.46$ recently obtained in \cite{Jammes}, but not as good as $\sqrt{8} \approx 2.828427$. In fact, one can show that
\[
\frac{16}{x^2} < \frac{1}{2\log(\cosh(x/4))}
\]
for every $x>0$, which means that our bound is better than Huber's for every value of $\sys(M)$. We will further decrease the lower bound on the systole sufficient to improve upon the inequality $\Nsmall(M) \leq 2g-2$ of Otal and Rosas in the next subsection.

\subsection{Numerical results for small systole}

Unsurprisingly, numerical optimization yields better results than \thmref{thm:small_asymp} when the systole is relatively small. For example, the resulting bounds show that $\Nsmall(M)<2g-2$ as soon as $\sys(M)\geq 2.317$ and that $\Nsmall(M)<g-1$ if $\sys(M) \geq 3.234$. A list of lower bounds on $\sys(M)$ and the upper bounds they imply on $\Nsmall(M)/(2g-2)$ is given in \tableref{table:nsmall}. The verification of these values is done in the ancillary file \texttt{verify\_nsmall.ipynb}. To produce the plot in \figref{fig:nsmall}, we used these values as well as the bounds produced at many other points and took a spline through this list of points. Thus, the plot itself is not rigorous, but the table is.

{\footnotesize
\begin{table}[htp] 
\centering
\caption{Lower bounds on $\sys(M)$ sufficient for $\Nsmall(M)/(2g-2)$ to be strictly smaller than given values.} \label{table:nsmall}
\begin{tabular}{|c|c|}
\hline
lower bound on $\sys(M)$ & strict upper bound on $\Nsmall(M)/(2g-2)$  \\
\hline
2.317 & 1   \\ 
3.234 & 1/2 \\ 
3.919 & 1/3  \\ 
4.486 & 1/4  \\
4.978 & 1/5  \\
5.409 & 1/6  \\
5.818 & 1/7  \\
6.180 & 1/8  \\
6.505 & 1/9  \\
6.894 & 1/10  \\
\hline
\end{tabular}
\end{table}
}

\subsection{Ramanujan surfaces}  \label{subsec:gaps}

Borrowing terminology from graph theory, we say that a hyperbolic surface $M$ of finite area is \emph{Ramanujan} if $\lambda_1(M) \geq 1/4$. We will also say that $M$ is \emph{strictly Ramanujan} if $\lambda_1(M) > 1/4$. Selberg's eigenvalue conjecture \cite{Selberg} states that all congruence covers of the modular curve are Ramanujan (hence such surfaces could also be called \emph{Selberg}). A related question is whether there exist closed Ramanujan surfaces in every genus (see \cite[Question 1.1]{Mondal} and \cite[Problem 10.4]{Wright}).  Thanks to the work of Hide and Magee \cite{HideMagee} (see also \cite{AnantharamanMonk2,polynomial_rate,MageePudervanHandel}), we now know that in large genus, there exist closed surfaces that are \emph{nearly Ramanujan} in the sense that their first eigenvalue is arbitrarily close to $1/4$.

Observe that
\[
\lambda_1(M)>L \ \Longleftrightarrow \  N_{[0,L]}(M) = 1  \ \Longleftrightarrow \ N_{[0,L]}(M) < 2.
\]
This means that one can prove lower bounds on $\lambda_1$ by bounding $N_{[0,L]}$ from above. In particular,
\[
\lambda_1(M)>\frac14 \ \Longleftrightarrow \  \Nsmall(M) < 2  \ \Longleftrightarrow \ \frac{\Nsmall(M)}{2g-2} < \frac{1}{g-1}
\]
so the values in \tableref{table:nsmall} give lower bounds on the systole that are sufficient for surfaces to be strictly Ramanujan in genus $2$ to $11$. However, in obtaining these values we discarded the term $1-\what{f}(i/2)$ appearing in \thmref{thm:nsmall}. The lower bounds on the systole (sufficient to be strictly Ramanujan) that we have obtained by taking this term into account are listed in \tableref{table:gaps} for $g$ between $2$ and $20$ and plotted in \figref{fig:gaps}. The corresponding ancillary file is \texttt{verify\_ramanujan.ipynb}. According to \tableref{table:sys}, there exist hyperbolic surfaces with systole larger than these bounds, hence strictly Ramanujan, in genus $2$ to $7$, $14$, and $17$. For these specific surfaces, we can increase $L$ further as long as $N_{[0,L]}<2$ to obtain improved lower bounds on $\lambda_1$ still based only on the systole. The resulting bounds are listed in Table \ref{table:lambda} except in genus $2$ to $4$ where better data was already available. The corresponding ancillary file is \texttt{verify\_ramanujan\_examples.ipynb}. These bounds are rigorous modulo proving that the lower bounds on the systole in \tableref{table:sys} are correct (as pointed out earlier, some are rigorous but not all).

{\footnotesize
\begin{table}[htp] 
\centering
\caption{Lower bounds on $\sys(M)$ sufficient for $\lambda_1(M)>1/4$.} \label{table:gaps}
\begin{tabular}{|c|c|l|}
\hline
genus & lower bound on $\sys(M)$ sufficient for $\lambda_1(M)>1/4$  & record systole\\
\hline
2 & 2.315  & 3.057141 \\ 
3 & 3.218 & 3.983304 \\ 
4 & 3.867 & 4.624499\\ 
5 & 4.380 & 4.91456 \\ 
6 & 4.803 & 5.109\\
7 & 5.168 & 5.796298 \\
8 & 5.482 &\\
9 & 5.760 & \\
10 & 6.010 & \\
11 & 6.236 & \\
12 & 6.443 & \\
13 & 6.632 & \\
14 & 6.808 &6.887905  \\
15 & 6.971 & \\
16 & 7.124 &\\
17 & 7.268 & 7.609407 \\
18 & 7.403 & \\
19 & 7.531 & \\
20 & 7.651 & \\
\hline
\end{tabular}
\end{table}
}

It seems very likely that surfaces with systole larger than the values listed in \tableref{table:gaps} exist in the remaining genera up to $20$ as well. However, our numerical experiments suggest that this method cannot prove that the next Hurwitz surface (of genus $118$) is Ramanujan using only its systole. Based on Figure \ref{fig:sys}, Figure \ref{fig:gaps}, and \thmref{thm:systole_asymp}, it seems reasonable to make the following conjecture.

\begin{conj}
There exist constants $c_1 > c_2$ such that the smallest upper bound on the systole that can be obtained from \thmref{thm:systole}  and the smallest lower bound on the systole sufficient to prove $\Nsmall < 2$ (equivalently $\lambda_1 > 1/4$) using \thmref{thm:nsmall} are of the form
\[
2 \log(g) + c_j + o(1)  \quad \text{as} \quad g \to \infty.
\]
\end{conj}

It is unknown if there exist closed hyperbolic surfaces with systole asymptotic to $r \log(g)$ for any $r> 4/3$, so even if the conjecture is true it is unlikely to be good enough to prove the existence of Ramanujan surfaces in large genus.

\bibliography{biblio}

\newcommand{\etalchar}[1]{$^{#1}$}
\begin{thebibliography}{FBGMPP23}

\bibitem[Ada02]{CharTables}
J.~Adams.
\newblock Character tables for {GL(2)}, {SL(2)}, {PGL(2)} and {PSL(2)} over a
  finite field.
\newblock Unpublished note,
  \href{http://www2.math.umd.edu/~jda/characters/characters.pdf}{www2.math.umd.edu/~jda/characters/characters.pdf},
  2002.

\bibitem[AM99]{AdamsMorgan}
C.~Adams and F.~Morgan.
\newblock Isoperimetric curves on hyperbolic surfaces.
\newblock {\em Proc. Amer. Math. Soc.}, 127(5):1347--1356, 1999.

\bibitem[AM23]{AnantharamanMonk}
N.~Anantharaman and L.~Monk.
\newblock {F}riedman--{R}amanujan functions in random hyperbolic geometry and
  application to spectral gaps.
\newblock Preprint, \href{https://arxiv.org/abs/2304.02678}{arXiv:2304.02678},
  2023.

\bibitem[AM25]{AnantharamanMonk2}
N.~Anantharaman and L.~Monk.
\newblock Friedman--{R}amanujan functions in random hyperbolic geometry and
  application to spectral gaps {II}.
\newblock Preprint, \href{https://arxiv.org/abs/2502.12268}{arXiv:2502.12268},
  2025.

\bibitem[AS64]{Handbook}
M.~Abramowitz and I.~A. Stegun.
\newblock {\em Handbook of mathematical functions with formulas, graphs, and
  mathematical tables}.
\newblock National Bureau of Standards Applied Mathematics Series, No. 55. U.
  S. Government Printing Office, Washington, D.C., 1964.
\newblock For sale by the Superintendent of Documents.

\bibitem[Bav86]{BavardKlein}
C.~Bavard.
\newblock In\'{e}galit\'{e} isosystolique pour la bouteille de {K}lein.
\newblock {\em Math. Ann.}, 274:439--441, 1986.

\bibitem[Bav96]{Bavard}
C.~Bavard.
\newblock Disques extr\'{e}maux et surfaces modulaires.
\newblock {\em Ann. Fac. Sci. Toulouse Math. (6)}, 5(2):191--202, 1996.

\bibitem[BBD88]{BBD}
P.~Buser, M.~Burger, and J.~Dodziuk.
\newblock Riemann surfaces of large genus and large {$\lambda_1$}.
\newblock In {\em Geometry and analysis on manifolds ({K}atata/{K}yoto, 1987)},
  volume 1339 of {\em Lecture Notes in Math.}, pages 54--63. Springer, Berlin,
  1988.

\bibitem[BC85]{BurgerColbois}
M.~Burger and B.~Colbois.
\newblock {\`A} propos de la multiplicit\'{e} de la premi\`ere valeur propre du
  laplacien d'une surface de {R}iemann.
\newblock {\em C. R. Acad. Sci. Paris S\'{e}r. I Math.}, 300(8):247--249, 1985.

\bibitem[BCP21]{diameter}
T.~Budzinski, N.~Curien, and B.~Petri.
\newblock On the minimal diameter of closed hyperbolic surfaces.
\newblock {\em Duke Math. J.}, 170(2):365--377, 2021.

\bibitem[BCP25]{cheeger_constant}
T.~Budzinski, N.~Curien, and B.~Petri.
\newblock On {C}heeger constants of hyperbolic surfaces.
\newblock {\em Invent. Math.}, 242(2):511--530, 2025.

\bibitem[Ben15]{Benson}
B.~A. Benson.
\newblock The {C}heeger constant, isoperimetric problems, and hyperbolic
  surfaces.
\newblock Preprint, \href{https://arxiv.org/abs/1509.08993}{arXiv:1509.08993},
  2015.

\bibitem[Ber16]{Bergeron}
N.~Bergeron.
\newblock {\em The spectrum of hyperbolic surfaces}.
\newblock Universitext. Springer, Cham; EDP Sciences, Les Ulis, 2016.
\newblock Appendix C by Valentin Blomer and Farrell Brumley, Translated from
  the 2011 French original by Brumley.

\bibitem[Bes80]{Besson}
G.~Besson.
\newblock Sur la multiplicit\'{e} de la premi\`ere valeur propre des surfaces
  riemanniennes.
\newblock {\em Ann. Inst. Fourier (Grenoble)}, 30(1):x, 109--128, 1980.

\bibitem[BLT21]{BensonLakelandThen}
B.~A. Benson, G.~S. Lakeland, and H.~Then.
\newblock Cheeger constants of hyperbolic reflection groups and {M}aass cusp
  forms of small eigenvalues.
\newblock {\em Proc. Amer. Math. Soc.}, 149(1):417--438, 2021.

\bibitem[Bon22]{bootstrap2}
J.~Bonifacio.
\newblock Bootstrapping closed hyperbolic surfaces.
\newblock {\em J. High Energy Phys.}, 93:Paper No. 093, 18, 2022.

\bibitem[BS94]{BuserSarnak}
P.~Buser and P.~Sarnak.
\newblock On the period matrix of a {R}iemann surface of large genus.
\newblock {\em Invent. Math.}, 117(1):27--56, 1994.
\newblock With an appendix by J. H. Conway and N. J. A. Sloane.

\bibitem[BS07]{BookerStrombergsson}
A.~R. Booker and A.~Str\"{o}mbergsson.
\newblock Numerical computations with the trace formula and the {S}elberg
  eigenvalue conjecture.
\newblock {\em J. Reine Angew. Math.}, 607:113--161, 2007.

\bibitem[BSV06]{BSV}
A.~R. Booker, A.~Str\"{o}mbergsson, and A.~Venkatesh.
\newblock Effective computation of {M}aass cusp forms.
\newblock {\em Int. Math. Res. Not.}, pages Art. ID 71281, 34, 2006.

\bibitem[BU83]{Bando}
S.~Bando and H.~Urakawa.
\newblock Generic properties of the eigenvalue of the {L}aplacian for compact
  {R}iemannian manifolds.
\newblock {\em Tohoku Math. J. (2)}, 35(2):155--172, 1983.

\bibitem[Bus77]{Buser14}
P.~Buser.
\newblock Riemannsche {F}l\"{a}chen mit {E}igenwerten in {$(0,$} {$1/4)$}.
\newblock {\em Comment. Math. Helv.}, 52(1):25--34, 1977.

\bibitem[Bus10]{Buser}
P.~Buser.
\newblock {\em Geometry and spectra of compact {R}iemann surfaces}.
\newblock Modern Birkh\"{a}user Classics. Birkh\"{a}user Boston, Inc., Boston,
  MA, 2010.
\newblock Reprint of the 1992 edition.

\bibitem[CB05]{Casamayou}
A.~Casamayou-Boucau.
\newblock Surfaces de {R}iemann parfaites en genre 4 et 6.
\newblock {\em Comment. Math. Helv.}, 80(3):455--482, 2005.

\bibitem[CCdV88]{CCV}
B.~Colbois and Y.~Colin~de Verdi\`ere.
\newblock Sur la multiplicit\'{e} de la premi\`ere valeur propre d'une surface
  de {R}iemann \`a courbure constante.
\newblock {\em Comment. Math. Helv.}, 63(2):194--208, 1988.

\bibitem[CdV86]{CdV86}
Y.~Colin~de Verdi\`ere.
\newblock Sur la multiplicit\'e de la premi\`ere valeur propre non nulle du
  laplacien.
\newblock {\em Comment. Math. Helv.}, 61(2):254--270, 1986.

\bibitem[CdV87]{CdV}
Y.~Colin~de Verdi\`ere.
\newblock Construction de laplaciens dont une partie finie du spectre est
  donn\'{e}e.
\newblock {\em Ann. Sci. \'{E}cole Norm. Sup. (4)}, 20:599--615, 1987.

\bibitem[CE03]{CohnElkies}
H.~Cohn and N.~Elkies.
\newblock New upper bounds on sphere packings. {I}.
\newblock {\em Ann. of Math. (2)}, 157(2):689--714, 2003.

\bibitem[Che70]{Cheeger}
J.~Cheeger.
\newblock A lower bound for the smallest eigenvalue of the {L}aplacian.
\newblock In {\em Problems in analysis ({S}ympos. in honor of {S}alomon
  {B}ochner, {P}rinceton {U}niv., {P}rinceton, {N}.{J}., 1969)}, pages
  195--199. Princeton Univ. Press, Princeton, N. J., 1970.

\bibitem[Che75]{Cheng}
S.~Y. Cheng.
\newblock Eigenvalue comparison theorems and its geometric applications.
\newblock {\em Math. Z.}, 143(3):289--297, 1975.

\bibitem[Che76]{ChengMultiplicity}
S.~Y. Cheng.
\newblock Eigenfunctions and nodal sets.
\newblock {\em Comment. Math. Helv.}, 51(1):43--55, 1976.

\bibitem[CK09]{CohnKumar}
H.~Cohn and A.~Kumar.
\newblock Optimality and uniqueness of the {L}eech lattice among lattices.
\newblock {\em Ann. of Math. (2)}, 170(3):1003--1050, 2009.

\bibitem[CKM{\etalchar{+}}17]{CKMRV}
H.~Cohn, A.~Kumar, S.~D. Miller, D.~Radchenko, and M.~Viazovska.
\newblock The sphere packing problem in dimension 24.
\newblock {\em Ann. of Math. (2)}, 185(3):1017--1033, 2017.

\bibitem[Con15]{ConderList}
M.~Conder.
\newblock Quotients of triangle groups acting on surfaces of genus 2 to 101.
\newblock
  \href{https://www.math.auckland.ac.nz/~conder/TriangleGroupQuotients101.txt}{math.auckland.ac.nz/~conder/TriangleGroupQuotients101.txt},
  2015.

\bibitem[Coo18]{Cook}
J.~Cook.
\newblock Properties of eigenvalues on {R}iemann surfaces with large symmetry
  groups.
\newblock PhD thesis, Loughborough University,
  \href{https://arxiv.org/abs/2108.11825}{arXiv:2108.11825}, 2018.

\bibitem[CS93]{ConwaySloane}
J.~H. Conway and N.~J.~A. Sloane.
\newblock {\em Sphere packings, lattices and groups}, volume 290 of {\em
  Grundlehren der mathematischen Wissenschaften [Fundamental Principles of
  Mathematical Sciences]}.
\newblock Springer-Verlag, New York, second edition, 1993.
\newblock With additional contributions by E. Bannai, R. E. Borcherds, J.
  Leech, S. P. Norton, A. M. Odlyzko, R. A. Parker, L. Queen and B. B. Venkov.

\bibitem[CZ14]{CohnZhao}
H.~Cohn and Y.~Zhao.
\newblock Sphere packing bounds via spherical codes.
\newblock {\em Duke Math. J.}, 163(10):1965--2002, 2014.

\bibitem[Del72]{Delsarte}
P.~Delsarte.
\newblock Bounds for unrestricted codes, by linear programming.
\newblock {\em Philips Res. Rep.}, 27:272--289, 1972.

\bibitem[DFM24]{KissingHigher}
C.~D\'oria, E.~M.~S. Freire, and P.~G.~P. Murillo.
\newblock Hyperbolic manifolds with a large number of systoles.
\newblock {\em Trans. Amer. Math. Soc.}, 377(2):1247--1271, 2024.

\bibitem[DGS77]{DGS}
P.~Delsarte, J.~M. Goethals, and J.~J. Seidel.
\newblock Spherical codes and designs.
\newblock {\em Geometriae Dedicata}, 6(3):363--388, 1977.

\bibitem[DK27]{Davis}
H.~T. Davis and W.~J. Kirkham.
\newblock A new table of the zeros of the {B}essel functions {$J_0(x)$} and
  {$J_1(x)$} with corresponding values of {$J_1(x)$} and {$J_0(x)$}.
\newblock {\em Bull. Amer. Math. Soc.}, 33(6):760--772, 1927.

\bibitem[DT00]{DT}
R.~Derby-Talbot.
\newblock Lengths of geodesics on {K}lein’s quartic quartic curve.
\newblock {\em Mathematical Sciences Technical Reports (MSTR)}, 98, 2000.

\bibitem[ESGJ06]{ElSoufi}
A.~El~Soufi, H.~Giacomini, and M.~Jazar.
\newblock A unique extremal metric for the least eigenvalue of the {L}aplacian
  on the {K}lein bottle.
\newblock {\em Duke Math. J.}, 135:181--202, 2006.

\bibitem[FBGMPP23]{counterexamples}
M.~Fortier~Bourque, É. Gruda-Mediavilla, B.~Petri, and M.~Pineault.
\newblock Two counterexamples to a conjecture of {C}olin de {V}erdi\`ere on
  multiplicity.
\newblock Preprint, \href{https://arxiv.org/abs/2312.03504}{arXiv:2312.03504},
  2023.

\bibitem[FBP22]{KissingManifolds}
M.~Fortier~Bourque and B.~Petri.
\newblock Kissing numbers of closed hyperbolic manifolds.
\newblock {\em Amer. J. Math.}, 144(4):1067--1085, 2022.

\bibitem[FBP24]{Klein}
M.~Fortier~Bourque and B.~Petri.
\newblock The {K}lein quartic maximizes the multiplicity of the first positive
  eigenvalue of the {L}aplacian.
\newblock {\em J. Differential Geom.}, 128(2):521--556, 2024.

\bibitem[FBR22]{FBRafi}
M.~Fortier~Bourque and K.~Rafi.
\newblock Local maxima of the systole function.
\newblock {\em J. Eur. Math. Soc. (JEMS)}, 24(2):623--668, 2022.

\bibitem[Gag80]{Gage}
M.~E. Gage.
\newblock Upper bounds for the first eigenvalue of the {L}aplace-{B}eltrami
  operator.
\newblock {\em Indiana Univ. Math. J.}, 29(6):897--912, 1980.

\bibitem[Gel63]{Gelfand}
I.~M. Gel{'}fand.
\newblock Automorphic functions and the theory of representations.
\newblock In {\em Proc. {I}nternat. {C}ongr. {M}athematicians ({S}tockholm,
  1962)}, pages 74--85. Inst. Mittag-Leffler, Djursholm, 1963.

\bibitem[GIT20]{Gorbachev}
D.~Gorbachev, V.~Ivanov, and S.~Tikhonov.
\newblock Uncertainty principles for eventually constant sign bandlimited
  functions.
\newblock {\em SIAM J. Math. Anal.}, 52(5):4751--4782, 2020.

\bibitem[GLMST21]{GLST}
C.~Gilmore, E.~Le~Masson, T.~Sahlsten, and J.~Thomas.
\newblock Short geodesic loops and {$L^p$} norms of eigenfunctions on large
  genus random surfaces.
\newblock {\em Geom. Funct. Anal.}, 31(1):62--110, 2021.

\bibitem[Ham01]{Hamenstadt}
U.~Hamenst\"{a}dt.
\newblock New examples of maximal surfaces.
\newblock {\em Enseign. Math. (2)}, 47(1-2):65--101, 2001.

\bibitem[Hej85]{Hejhal}
D.~A. Hejhal.
\newblock A classical approach to a well-known spectral correspondence on
  quaternion groups.
\newblock In {\em Number theory ({N}ew {Y}ork, 1983--84)}, volume 1135 of {\em
  Lecture Notes in Math.}, pages 127--196. Springer, Berlin, 1985.

\bibitem[Her70]{Hersch}
J.~Hersch.
\newblock Quatre propri\'{e}t\'{e}s isop\'{e}rim\'{e}triques de membranes
  sph\'{e}riques homog\`enes.
\newblock {\em C. R. Acad. Sci. Paris S\'{e}r. A-B}, 270:A1645--A1648, 1970.

\bibitem[HM23]{HideMagee}
W.~Hide and M.~Magee.
\newblock Near optimal spectral gaps for hyperbolic surfaces.
\newblock {\em Ann. of Math. (2)}, 198(2):791--824, 2023.

\bibitem[HMT25]{polynomial_rate}
W.~Hide, D.~Macera, and J.~Thomas.
\newblock Spectral gap with polynomial rate for {W}eil--{P}etersson random
  surfaces.
\newblock Preprint, \href{https://arxiv.org/abs/2508.14874}{arXiv:2508.14874},
  2025.

\bibitem[Hub76]{HuberSmallEigs}
H.~Huber.
\newblock \"{U}ber die {E}igenwerte des {L}aplace-{O}perators auf kompakten
  {R}iemannschen {F}l\"{a}chen.
\newblock {\em Comment. Math. Helv.}, 51(2):215--231, 1976.

\bibitem[Hur92]{Hurwitz}
A.~Hurwitz.
\newblock Ueber algebraische {G}ebilde mit eindeutigen {T}ransformationen in
  sich.
\newblock {\em Math. Ann.}, 41(3):403--442, 1892.

\bibitem[Hux85]{Huxley}
M.~N. Huxley.
\newblock Introduction to {K}loostermania.
\newblock In {\em Elementary and analytic theory of numbers ({W}arsaw, 1982)},
  volume~17 of {\em Banach Center Publ.}, pages 217--306. PWN, Warsaw, 1985.

\bibitem[Jam21]{Jammes}
P.~Jammes.
\newblock Systole and small eigenvalues of hyperbolic surfaces.
\newblock Preprint, \href{https://arxiv.org/abs/2111.08985}{arXiv:2111.08985},
  2021.

\bibitem[Jen84]{Jenni}
F.~Jenni.
\newblock \"{U}ber den ersten {E}igenwert des {L}aplace-{O}perators auf
  ausgew\"{a}hlten {B}eispielen kompakter {R}iemannscher {F}l\"{a}chen.
\newblock {\em Comment. Math. Helv.}, 59(2):193--203, 1984.

\bibitem[JL70]{JacquetLanglands}
H.~Jacquet and R.~P. Langlands.
\newblock {\em Automorphic forms on {${\rm GL}(2)$}}.
\newblock Lecture Notes in Mathematics, Vol. 114. Springer-Verlag, Berlin-New
  York, 1970.

\bibitem[JLN{\etalchar{+}}05]{Jakobson}
D.~Jakobson, M.~Levitin, N.~Nadirashvili, N.~Nigam, and I.~Polterovich.
\newblock How large can the first eigenvalue be on a surface of genus two?
\newblock {\em Int. Math. Res. Not.}, 2005(63):3967--3985, 2005.

\bibitem[Joh17]{Johansson}
F.~Johansson.
\newblock Arb: efficient arbitrary-precision midpoint-radius interval
  arithmetic.
\newblock {\em IEEE Transactions on Computers}, 66:1281--1292, 2017.

\bibitem[KMcP24]{bootstrap}
P.~Kravchuk, D.~Maz\'a\v~c, and S.~Pal.
\newblock Automorphic spectra and the conformal bootstrap.
\newblock {\em Commun. Am. Math. Soc.}, 4:1--63, 2024.

\bibitem[KSV07]{KSV}
M.~G. Katz, M.~Schaps, and U.~Vishne.
\newblock Logarithmic growth of systole of arithmetic {R}iemann surfaces along
  congruence subgroups.
\newblock {\em J. Differential Geom.}, 76(3):399--422, 2007.

\bibitem[Lee23]{Lee}
C.~Lee.
\newblock Spectrum of the {L}aplacian on the {F}ricke-{M}acbeath surface.
\newblock Preprint, \href{https://arxiv.org/abs/2311.02632}{arXiv:2311.02632},
  2023.

\bibitem[Lev79]{Lev}
V.~I. Leven{\v{s}}te{\u\i}n.
\newblock Boundaries for packings in {$n$}-dimensional {E}uclidean space.
\newblock {\em Soviet Math. Dokl.}, 20:417--421, 1979.

\bibitem[LM24]{LetrouitMachado}
Cyril Letrouit and Simon Machado.
\newblock Maximal multiplicity of {L}aplacian eigenvalues in negatively curved
  surfaces.
\newblock {\em Geom. Funct. Anal.}, 34(5):1609--1645, 2024.

\bibitem[LW24]{LipnowskiWright}
M.~Lipnowski and A.~Wright.
\newblock Towards optimal spectral gaps in large genus.
\newblock {\em Ann. Probab.}, 52(2):545--575, 2024.

\bibitem[LY82]{LiYau}
P.~Li and S.-T. Yau.
\newblock A new conformal invariant and its applications to the {W}illmore
  conjecture and the first eigenvalue of surfaces.
\newblock {\em Invent. Math.}, 69:269--291, 1982.

\bibitem[Miy06]{Miyake}
T.~Miyake.
\newblock {\em Modular forms}.
\newblock Springer Monographs in Mathematics. Springer-Verlag, Berlin, english
  edition, 2006.
\newblock Translated from the 1976 Japanese original by Yoshitaka Maeda.

\bibitem[MNP22]{covers}
M.~Magee, F.~Naud, and D.~Puder.
\newblock A random cover of a compact hyperbolic surface has relative spectral
  gap {$\frac{3}{16}-\varepsilon$}.
\newblock {\em Geom. Funct. Anal.}, 32(3):595--661, 2022.

\bibitem[Mon15]{Mondal}
S.~Mondal.
\newblock On largeness and multiplicity of the first eigenvalue of finite area
  hyperbolic surfaces.
\newblock {\em Math. Z.}, 281(1-2):333--348, 2015.

\bibitem[Mon22]{Monk}
L.~Monk.
\newblock Benjamini-{S}chramm convergence and spectra of random hyperbolic
  surfaces of high genus.
\newblock {\em Anal. PDE}, 15(3):727--752, 2022.

\bibitem[MPvH25]{MageePudervanHandel}
M.~Magee, D.~Puder, and R.~van Handel.
\newblock Strong convergence of uniformly random permutation representations of
  surface groups.
\newblock Preprint, \href{https://arxiv.org/abs/2504.08988}{arXiv:2504.08988},
  2025.

\bibitem[MRT14]{MRT}
J.~Malestein, I.~Rivin, and L.~Theran.
\newblock Topological designs.
\newblock {\em Geom. Dedicata}, 168:221--233, 2014.

\bibitem[Nad88]{NadirashviliMultiplicite}
N.~S. Nadirashvili.
\newblock Multiple eigenvalues of the {L}aplace operator.
\newblock {\em Math. USSR-Sb.}, 61(1):225--238, 1988.

\bibitem[Nad96]{Nadirashvili}
N.~Nadirashvili.
\newblock Berger's isoperimetric problem and minimal immersions of surfaces.
\newblock {\em Geom. Funct. Anal.}, 6:877--897, 1996.

\bibitem[NS19]{NayataniShoda}
S.~Nayatani and T.~Shoda.
\newblock Metrics on a closed surface of genus two which maximize the first
  eigenvalue of the {L}aplacian.
\newblock {\em C. R. Math. Acad. Sci. Paris}, 357:84--98, 2019.

\bibitem[OLBC10]{NIST}
F.~W.~J. Olver, D.~W. Lozier, R.~F. Boisvert, and C.~W. Clark, editors.
\newblock {\em N{IST} handbook of mathematical functions}.
\newblock U.S. Department of Commerce, National Institute of Standards and
  Technology, Washington, DC; Cambridge University Press, Cambridge, 2010.
\newblock With 1 CD-ROM (Windows, Macintosh and UNIX).

\bibitem[OR09]{OtalRosas}
J.-P. Otal and E.~Rosas.
\newblock Pour toute surface hyperbolique de genre {$g,\ \lambda_{2g-2}>1/4$}.
\newblock {\em Duke Math. J.}, 150(1):101--115, 2009.

\bibitem[Ota08]{Otal}
J.-P. Otal.
\newblock Three topological properties of small eigenfunctions on hyperbolic
  surfaces.
\newblock In {\em Geometry and dynamics of groups and spaces}, volume 265 of
  {\em Progr. Math.}, pages 685--695. Birkh\"{a}user, Basel, 2008.

\bibitem[Par13]{Parlier}
H.~Parlier.
\newblock Kissing numbers for surfaces.
\newblock {\em J. Topol.}, 6(3):777--791, 2013.

\bibitem[Pin24]{Pineault}
M.~Pineault.
\newblock Multiplicité des valeurs propres du laplacien sur les surfaces
  hyperboliques triangulaires.
\newblock Master's thesis, Université de Montréal,
  \href{http://hdl.handle.net/1866/40302}{DOI:10.71781/15107}, 2024.

\bibitem[Pu52]{Pu}
P.~M. Pu.
\newblock Some inequalities in certain nonorientable {R}iemannian manifolds.
\newblock {\em Pacific J. Math.}, 2:55--71, 1952.

\bibitem[Sav09]{Savo}
A.~Savo.
\newblock On the lowest eigenvalue of the {H}odge {L}aplacian on compact,
  negatively curved domains.
\newblock {\em Ann. Global Anal. Geom.}, 35(1):39--62, 2009.

\bibitem[Sch93]{SchmutzMaxima}
P.~Schmutz.
\newblock {R}iemann surfaces with shortest geodesic of maximal length.
\newblock {\em Geom. Funct. Anal.}, 3:564--631, 1993.

\bibitem[Sch94]{SchmutzKissing}
P.~Schmutz.
\newblock Systoles on {R}iemann surfaces.
\newblock {\em Manuscripta Math.}, 85(3-4):429--447, 1994.

\bibitem[Sel65]{Selberg}
A.~Selberg.
\newblock On the estimation of {F}ourier coefficients of modular forms.
\newblock In {\em Proc. {S}ympos. {P}ure {M}ath., {V}ol. {VIII}}, pages 1--15.
  Amer. Math. Soc., Providence, R.I., 1965.

\bibitem[S{\'e}v02]{Sevennec}
B.~S{\'e}vennec.
\newblock Multiplicity of the second {S}chr\"{o}dinger eigenvalue on closed
  surfaces.
\newblock {\em Math. Ann.}, 324(1):195--211, 2002.

\bibitem[Sin74]{Singerman}
D.~Singerman.
\newblock Symmetries of {R}iemann surfaces with large automorphism group.
\newblock {\em Math. Ann.}, 210:17--32, 1974.

\bibitem[SS97]{Schmutz43}
P.~Schmutz~Schaller.
\newblock Extremal {R}iemann surfaces with a large number of systoles.
\newblock In {\em Extremal {R}iemann surfaces ({S}an {F}rancisco, {CA}, 1995)},
  volume 201 of {\em Contemp. Math.}, pages 9--19. Amer. Math. Soc.,
  Providence, RI, 1997.

\bibitem[SS22]{ScheinShoan}
M.~M. Schein and A.~Shoan.
\newblock Systolic length of triangular modular curves.
\newblock {\em Journal of Number Theory}, 239:462--488, 2022.

\bibitem[SU13]{StrohmaierUski}
A.~Strohmaier and V.~Uski.
\newblock An algorithm for the computation of eigenvalues, spectral zeta
  functions and zeta-determinants on hyperbolic surfaces.
\newblock {\em Comm. Math. Phys.}, 317:827--869, 2013.

\bibitem[{The}21]{sagemath}
{The Sage Developers}.
\newblock {\em {S}ageMath, the {S}age {M}athematics {S}oftware {S}ystem
  ({V}ersion 9.3)}, 2021.
\newblock \url{https://www.sagemath.org}.

\bibitem[Via17]{Viazovska}
M.~S. Viazovska.
\newblock The sphere packing problem in dimension 8.
\newblock {\em Ann. of Math. (2)}, 185(3):991--1015, 2017.

\bibitem[Vog03]{VogelerThesis}
R.~Vogeler.
\newblock {\em On the geometry of {H}urwitz surfaces}.
\newblock ProQuest LLC, Ann Arbor, MI, 2003.
\newblock Thesis (Ph.D.)--The Florida State University.

\bibitem[Wat95]{Watson}
G.~N. Watson.
\newblock {\em A treatise on the theory of {B}essel functions}.
\newblock Cambridge Mathematical Library. Cambridge University Press,
  Cambridge, 1995.
\newblock Reprint of the second (1944) edition.

\bibitem[Woo01]{Woods}
K.~Woods.
\newblock Lengths of systoles on tileable hyperbolic surfaces.
\newblock {\em Mathematical Sciences Technical Reports (MSTR)}, 102, 2001.

\bibitem[Wri20]{Wright}
A.~Wright.
\newblock A tour through {M}irzakhani's work on moduli spaces of {R}iemann
  surfaces.
\newblock {\em Bull. Amer. Math. Soc. (N.S.)}, 57(3):359--408, 2020.

\bibitem[WX22]{WuXue}
Y.~Wu and Y.~Xue.
\newblock Random hyperbolic surfaces of large genus have first eigenvalues
  greater than {$\frac{3}{16}-\epsilon$}.
\newblock {\em Geom. Funct. Anal.}, 32(2):340--410, 2022.

\bibitem[YY80]{YangYau}
P.~C. Yang and S.~T. Yau.
\newblock Eigenvalues of the {L}aplacian of compact {R}iemann surfaces and
  minimal submanifolds.
\newblock {\em Ann. Scuola Norm. Sup. Pisa Cl. Sci. (4)}, 7(1):55--63, 1980.

\end{thebibliography}
\bibliographystyle{alpha}
\end{document}